\documentclass[11pt,leqno,amscd,amssymb,verbatim, url]{amsart}
\usepackage{amsfonts,amssymb,graphicx}
\usepackage{amsmath,amscd}
\usepackage{hyperref}

\newtheorem{thm}{Theorem}[subsection] 

\makeatletter\let\c@fact\c@theorem\makeatother\newtheorem{lem}[thm]{Lemma}
\newtheorem{cor}[thm]{Corollary}

\newtheorem{prop}[thm]{Proposition}
\newtheorem{assumptions}[thm]{Assumptions}
\theoremstyle{definition}

\newtheorem{rem}[thm]{Remark}

\numberwithin{equation}{subsection}
\numberwithin{thm}{section}
\newcommand{\wdelta}{{\widetilde{\mathfrak d}}}
\newcommand{\wnabla}{{\widetilde\nabla}}
\newcommand{\wrnabla}{{{\widetilde\nabla}}_{\text{\rm rad}}}
\newcommand{\wF}{{\widetilde F}}

\newcommand{\grhom}{\text{\rm hom}}

\newcommand{\grHom}{\text{\rm hom}}
\newcommand{\grExt}{\text{\rm ext}}

\newcommand{\wfa}{\widetilde{\mathfrak a}}
\newcommand{\wrDelta}{{\widetilde{\Delta}}^{\text{\rm red}}}

\newcommand{\wX}{{\widetilde{X}}}

\newcommand{\wgr}{\widetilde{\text{\rm gr}}\,}

\newcommand{\wsoc}{\widetilde{\text{\rm soc}}}
\newcommand{\wrad}{\widetilde{\text{\rm rad}}\,}

\newcommand\Xreg{X_{\text{\rm reg}}(T)_+}

\newcommand{\whQ}{\widehat Q}

\newcommand{\gr}{\text{\rm gr}}

\newcommand{\sA}{{\mathcal A}}

\newcommand{\Ext}{{\text{\rm Ext}}}

\newcommand{\fa}{{\mathfrak a}}

\newcommand{\Amod}{A\mbox{--mod}}

\newcommand{\Hom}{\text{\rm Hom}}

\newcommand{\ch}{\operatorname{ch}}

\newcommand{\soc}{\operatorname{soc}}
\newcommand{\sO}{{\mathcal{O}}}
\newcommand{\Uzmod}{U_\zeta{\text{\rm --mod}}}
\newcommand{\rad}{\operatorname{rad}}

\newcommand{\Dist}{\operatorname{Dist}}

\newcommand{\rDelta}{\Delta^{\text{\rm red}}}
\newcommand{\rnabla}{\nabla_{\text{\rm red}}}
\newcommand{\wM}{{\widetilde{M}}}
\newcommand{\wA}{{\widetilde{A}}}
\newcommand{\wB}{{\widetilde{B}}}

\newcommand{\wU}{{\widetilde U}}
\newcommand{\wDelta}{{\widetilde{\Delta}}}

\newcommand{\wN}{{\widetilde{N}}}
\newcommand{\wD}{{\widetilde{D}}}
\newcommand{\wP}{{\widetilde{P}}}
\newcommand{\wQ}{{\widetilde{Q}}}

\newcommand{\Jan}{\Gamma_{\text{\rm Jan}}}

\newcommand{\wL}{{\widetilde{L}}}

\newcommand{\wY}{{\widetilde{Y}}}

\newcommand{\St}{{\text{\rm St}}}

\newcommand{\Gmod}{{G{\text{\rm --mod}}}}

\newcommand{\blist}{\begin{list}{\rom{(\roman{enumi})}}{\setlength
{\leftmarg in}{0em} \setlength{\itemindent}{7ex}
\setlength{\labelsep}{2ex}\setlength{\listparindent}{\parindent}
\usecounter{enumi}}}
\newcommand{\elist}{\end{list}}


\begin{document}

\begin{abstract}This paper considers Weyl modules for a simple, simply connected algebraic group $G$
over an algebraically closed field $k$ of positive characteristic $p\not=2$. The main result proves, if $p\geq 2h-2$ (where $h$ is the Coxeter number) and if the Lusztig character formula holds for all (irreducible modules with) $p$-regular $p$-restricted highest weights, then any Weyl module $\Delta(\lambda)$ has a $\Delta^p$-filtration, namely, a filtration
with sections of the form $\Delta^p(\mu_0+p\mu_1):=L(\mu_0)\otimes\Delta(\mu_1)^{[1]}$, where $\mu_0$ is $p$-restricted and $\mu_1$
is arbitrary dominant. In case the highest weight $\lambda$ of the Weyl module $\Delta(\lambda)$ is $p$-regular, the $p$-filtration is
compatible with the $G_1$-radical series of the module.  The problem of showing that Weyl modules
have $\Delta^p$-filtrations was first proposed  by Jantzen in 1980. The proof in this paper is based on
new methods involving ``forced gradings" arising from orders associated to quantum enveloping
algebras.  A new $\Ext^1$-criterion is proved for $\Delta^p$-filtrations, but only in the context of such
forced gradings. Finally, these results have already had applications to the $G$-module structure of Ext-groups for
the restricted enveloping algebra of $G$.  \end{abstract}

 \title[On $p$-filtrations of Weyl modules]{On $p$-filtrations of Weyl modules}\author{Brian J. Parshall}
\address{Department of Mathematics \\
University of Virginia\\
Charlottesville, VA 22903} \email{bjp8w@virginia.edu {\text{\rm
(Parshall)}}}
\author{Leonard L. Scott}
\address{Department of Mathematics \\
University of Virginia\\
Charlottesville, VA 22903} \email{lls2l@virginia.edu {\text{\rm
(Scott)}}}

\thanks{Research supported in part by the National Science
Foundation}
\maketitle

\section{Introduction} Let $G$ be a simple, simply connected algebraic group defined and split
over a prime field ${\mathbb F}_p$. For a dominant weight $\lambda$, let $\Delta(\lambda)$
be the Weyl module for $G$ of highest weight $\lambda$.  It has dimension and character given
by Weyl's dimension and character formulas, respectively,   reflecting
the fact that it arises through a standard reduction mod $p$ process from a minimal admissible lattice in the
irreducible module $L_{\mathbb C}(\lambda)$ for the complex simple Lie algebra of the same type as $G$. Given a dominant
weight $\mu$, write $\mu=\mu_0+p\mu_1$, where $\mu_0$ (resp., $\mu_1$) is a $p$-restricted (resp., arbitrary) dominant weight.
Put $\Delta^p(\mu):=L(\mu_0)\otimes \Delta(\mu_1)^{[1]}$, where, in general, given a rational $G$-module $V$, $V^{[1]}$
denotes the rational $G$-module obtained by making $G$ act on $V$ through the Frobenius morphism $F:G\to G$. An old question, going back to Janzten's 1980 Crelle paper \cite{Jan}, asks if every Weyl
module $\Delta(\lambda)$ has a  $\Delta^p$-filtration, i.~e., a filtration by $G$-submodules with sections isomorphic to modules of
the form $\Delta^p(\mu)$.   In his paper, Jantzen gave a positive answer to this question
when $\lambda$ is ``generic" in the sense of being sufficiently far from the walls, and in the cases when $G\cong SL_2$
or $SL_3$. See \cite[Thm. 3.8, Rem. 3.8(2), 3.13]{Jan}. 
  (There is also the dual notion of a  $\nabla_p$-filtration, using modules $\nabla_p(\mu):=L(\mu_0)\otimes\nabla(\mu_1)^{[1]}$; see (\ref{pmodule}) below.) 

 Let $p\geq 2h-2$ be an odd prime\footnote{The assumption that $p$ is odd is largely a convenience, since $p=2\geq 2h-2$ implies $G=SL_2$. In this case, the Jantzen question on $p$-filtrations has an easy and positive answer.},
 where $h=(\rho,\alpha_0^\vee)+1$ is the Coxeter number of $G$. Theorem \ref{MainTheorem}, the main theorem of this paper, establishes that,
 if the Lusztig character formula
 holds for all $p$-restricted irreducible modules, then each Weyl module $\Delta(\lambda)$,  with $\lambda$
 any dominant weight, has a
$\Delta^p$-filtration.  A surprising feature of the proof is that, in the case of $p$-regular $\lambda$, the $\Delta^p$-filtration of $\Delta(\lambda)$ is  consistent with its $G_1$-radical filtration.\footnote{The passage from $p$-regular to non-$p$-regular highest weights is achieved though
 an elementary use of Jantzen translation functors. Translation preserves $\Delta^p$-filtrations, though the $G_1$-radical filtration might be lost (an open question).} Moreover, it seems necessary in the argument to prove the theorem in this form. Perhaps even more
 surprising, it also seems necessary to approach these Weyl modules and their $G_1$-radical series through
 their quantum analogues. The key point hinges on a new $\Ext^1$-criterion (see Theorem
  \ref{firstmainnum}), important in its own right,
  for the existence of a
 $\Delta^p$-filtration of a Weyl module by passing to a context of graded modules, and with the grading constructed from a filtration in a quantum lifting. Such liftings may not exist for all modules, but they do exist for Weyl modules. It is important to emphasize that the $\Ext^1$-criterion is available only in a graded context
 afforded by such liftings. 
 
 Let $U_\zeta$ be the (Lusztig) quantum enveloping algebra (having the same root type as
 $G$) at a primitive $p$th root of unity $\zeta$. The proof of Theorem \ref{MainTheorem} makes essential
 use of new results in integral representation theory (developed in \cite{PS10}) to link the representation theory
 of  $U_\zeta$ and that of $G$. (Here the integral algebras are taken over a suitable DVR $\sO$.)
 In particular, the irreducible $U_\zeta$-module $L_\zeta(\lambda)$ of highest weight $\lambda$ has
 two natural ``reductions mod $p$", denoted $\rDelta(\lambda)$ and $\rnabla(\lambda)$. These modules, 
 defined in (\ref{reduced}), 
 have properties somewhat similar to Weyl and dual Weyl modules, and their characters and dimensions
 are given in terms of Kazhdan-Lusztig polynomials.  Assume $p>h$ and the Lusztig character formula for $G$ holds\footnote{See fn. 6 for further discussion.} for irreducible modules having $p$-restricted highest weights. It follows that $\rDelta(\lambda)=\Delta^p(\lambda)$ for any dominant weight $\lambda$; see Lemma \ref{LCF}(c).  Surprisingly,  it is often easier and more natural to work with the $\rDelta$-modules, coming from the
 quantum enveloping algebra $U_\zeta$, rather than the intrinsic $G$-modules $\Delta^p(\lambda)$.\footnote{As we
 show below, any Weyl module $\rDelta(\lambda)$ has a $\Delta^p$-filtration, at least for $p\geq 2h-2$, so, in order to show that $\Delta(\lambda)$ has a $\Delta^p$-filtration, it suffices to show that $\Delta(\lambda)$
 has a $\rDelta$-filtration.  However, there is an example, when $p=2$, 
 due to Will Turner (unpublished), of a Weyl module which does not have an $\rDelta$-filtration. In this case,
 $G=SL_5$ and $h=5$. We have checked that it does have a $\Delta^p$-filtration. We do not know of examples of Weyl modules (for any prime) that do not have a $\Delta^p$-filtration. }  Also, see Theorem \ref{thm7.1} and Remark \ref{rem5.2} for positive results involving  the modules $\rDelta(\lambda)$, assuming
  the Lusztig character formula holds for some $p$-restricted irreducible modules.
 
 An essential feature of our method here, which builds on ideas developed in \cite{PS10}, involves a ``forced
 grading" $\gr\wA$ of the integral quasi-hereditary algebras $\wA$ which connect the representation theory of $U_\zeta$
 and $G$. Thus, a main result \cite[Thm. 5.3]{PS10} shows that $\gr\wA$ is itself quasi-hereditary. Then
 Corollary 3.2  establishes that the (integral) algebra $\wB:=(\gr\wA)_0$ obtained from the
 grade 0 piece of $\gr\wA$ is quasi-hereditary with standard modules $\wrDelta(\lambda)$, a lift of
 $\rDelta(\lambda)$ to an $\sO$-lattice. In this way, the modules $\rDelta(\lambda)$ arise
 naturally. At this point, there is no assumption that the Lusztig character formula holds. However, it enters
 into Lemma \ref{preliminaryLCF}(c) and thus plays a role in the establishment of Theorem \ref{thm4.4} and its corollary
 Theorem \ref{firstmainnum}, the  
$\Ext^1$-criterion mentioned above.  Using several further technical results (for example, it is required to prove that Weyl modules have
``tight liftings"; see definition immediately after the proof of Theorem \ref{secondmainnum}), we obtain the main results in \S5. 

The $\Ext^1$-criterion given in Theorem \ref{firstmainnum} suggests there should be $\rDelta$-filtration results for
appropriately defined syzygies of a module $\rDelta(\lambda)$ and graded versions of the Weyl
modules $\Delta(\lambda)$. This is indeed the case, but the arguments require additional considerations
involving Koszul algebras, both using and improving upon the results of \cite{PS9}. This work is
presented in a sequel \cite{PS15} to this paper. In particular, it proved, under the hypotheses of \S 5, that $\Ext^n_{G_1}(\rDelta(\lambda),
\rnabla(\mu))^{[1]}$ has a good filtration for all $p$-regular dominant weights $\lambda,\mu$. Here $G_1$ is the first Frobenius kernel of $G$.
 
 As noted above, the problem of finding a $\Delta^p$-filtration for Weyl modules was first proposed by
 Jantzen \cite[p. 173]{Jan}, who obtained several positive results in his paper. In addition, Jantzen observed
 that, for $p\geq 2h-2$, the character of the Weyl module could be written as a non-negative linear
 combination of the characters of modules $\Delta^p(\mu)$.  Under the $p\geq 2h-2$ assumption, Andersen proposed in \cite{A1} a
 positive solution. While providing worthwhile connections to other problems
 and conjectures, his proposed solution was later withdrawn \cite{A2}. To our knowledge, there has  been no
 further progress since Andersen's work, though Parker obtained in \cite{P2} an interesting extension
 of Jantzen's $SL_3$ work to quantum  $GL_3$ at a root of unity $\ell\geq 2$ over fields of arbitrary
  characteristic.  
\medskip\medskip

 \begin{center} {\bf Notation}\end{center}
 
 The following notation related to a $p$-modular system will be used throughout the paper.

\medskip
\begin{enumerate}

\item $(K,\sO,k)$: $p$-modular system. Thus, $\sO$ is a DVR with maximal ideal ${\mathfrak m}=(\pi)$,  fraction field $K$, and residue
field $k$. An $\sO$-lattice $\wM$ is, by definition, an $\sO$-module which is free and of finite rank.

\item  $\wA$: $\sO$-algebra which is finite and free over $\sO$. Let $\wA_K:= K\otimes_{\sO}\wA$ and
$A:=k\otimes_{\sO}\wA$. More generally, if $\wM$ is an $\wA$-module, put $\wM_K:=K\otimes_{\sO}\wM$
and $M= \wM_k:=k\otimes_{\sO}\wM$. Sometimes, $\wM_k$ is also denoted $\overline{\wM}$. Often
$\wM$ will be finite and free over $\sO$---namely, a {\it lattice} for $\sO$ (or $\wA$).

\item Let $\wM$ be an $\wA$-lattice. Put $\wrad^n\wM:=\wM\cap \rad^n\wM_K$, where $\rad^n\wM_K$ denotes the $n$th-radical of 
the $\wA_K$-module $\wM_K$. Of course, $\rad^n\wM_K=(\rad^n\wA_K)\wM_K$.

Dually, let $\widetilde\soc^{-n}\wM:=\soc^{-n}\wM_K\cap\wM$, $n=0,1,\cdots$, where $\{\soc^{-n}\wM_K\}_n$ 
is the socle series of $\wM_K$.

\item Again, let $\wM$ is an $\wA$-lattice. $\gr \wM:=\bigoplus_{n\geq 0}\wrad^n\wM/\wrad^{n+1}\wM$, viewed as a (positively) graded module for the $\sO$-algebra $\gr\wA:=\bigoplus_{n\geq 0}\wrad^n\wA/\wrad^{n+1}\wA$.
Notice that $\gr\wA$ is an $\sO$-lattice, and that   $\gr \wM$ is 
a $\gr\wA$-lattice.

Dually, let $\gr^\diamond\wM:=\bigoplus_{n\geq 0}\widetilde\soc^{-n}\wM/\widetilde\soc^{-n+1}\wM$, regarded
as a {\it negatively} graded $\gr\wA$-lattice. Observe that, taking $\sO$-duals, $(\gr^\diamond\wM)^*=(\gr \wM^*)$
as $(\gr\wA)^{\text{\rm op}}$-lattice.

\item We say that a $\wA$-lattice $\wM$ is {\it $\wA$-tight} (or just {\it tight}, if $\wA$ is clear from context) if
\begin{equation}\label{rad}(\wrad^n\wA)\wM=\wrad^n\wM,\quad\forall n\geq 0.
\end{equation}
Clearly, if $\wM$ is also $\wA$-projective, then it is tight. (We will see that many other lattices can be tight.)

\item Now let $\wfa$ be an $\sO$-subalgebra of $\wA$. Then items (2)--(5) all make perfectly good sense
using $\wfa$ in place of $\wA$. If $\wM$ is an $\wA$-lattice, then it is an $\wfa$-lattice. 
 In our applications later, it will usually be the case that 
$(\rad^n\wfa_K)\wA_K=\rad^n\wA_K$, for all $n\geq 0$; see (\ref{radicals}). In that case, if $\wM$ is an $\wA$-lattice,
then $\wrad^n\wM$ can be constructed viewing $\wM$ as an $\wA$-lattice or as an $\wfa$-lattice. Both
constructions lead to identical $\sO$-modules.
Ambiguities of a formal nature may still arise as to whether it is more appropriate to use $\wfa$ or $\wA$, but are
generally resolved by context. Similar remarks apply for $\gr\wM$. Often the $\wA$-tightness
of $\wM$ is the same as its $\wfa$-tightness; see Corollary \ref{tightcor}.
\end{enumerate}

\medskip
 
 Now suppose, for the moment, that $A$ is a positively graded algebra. Given graded $A$-modules $M,N$, the {\it graded}
 $\Ext^n$-groups, will be denoted $\grExt^n_A(M,N)$, $n=0,1,\cdots$.  (These are the $\Ext^n$-groups
 computed in the category of graded $A$-modules.) When $n=0$, the space of
 homomorphisms $M\to N$ preserving grades is denoted $\grHom_A(M,N)= \grExt^0_A(M,N)$. The 
 ordinary extension groups are $\Ext^n_A(M,N)$, where $M,N$ need not be graded. If $M,N$ are graded,
 $\Ext^n_A(M,N)=\bigoplus_{r\in\mathbb Z}\grExt^n_A(M,N(r))$, $n=0,1,2,\cdots$, where $N(r)$ is the graded module
 obtained from $N$ setting $N(r)_i:=N(i-r)$.

\section{Algebraic and quantum groups}
\setcounter{equation}{0}

Let $G$ be a simple, simply connected algebraic group defined and split over ${\mathbb F}_p$, where $p$ is a prime integer.
Fix a maximal split torus $T$ of $G$ and a Borel subgroup $B\supset T$.  Let $R$ be the root system of
$T$, and let $S$ be the set of simple roots corresponding to the Borel subgroup $B^+$ opposite to $B$.
Throughout this paper, we will, with a few exceptions, carefully follow the (very standard) notation laid down in \cite[pp. 569--572]{JanB} regarding $G$ and its representation theory. Two exceptions are as follows.  Given
$\lambda\in X(T)_+$, the Weyl module for $G$ of highest weight $\lambda$ will be denoted $\Delta(\lambda)$,
rather than $V(\lambda)$ as in \cite{JanB}.
Also, let $\nabla(\lambda)=\Delta(\lambda^\star)^* $ be the ``dual Weyl module" of highest weight $\lambda$.
Here $\lambda^\star:=-w_0(\lambda)$, where $w_0$ is the longest element in the Weyl group of $G$. In the notation of \cite{JanB}, $\nabla(\lambda)$ is denoted 
$H^0(\lambda)$.

A weight $\lambda\in X(T)$ is called $p$-regular (or just ``regular" since the prime $p$ will be understood as fixed)
provided that, for any root $\alpha$, $(\lambda+\rho,\alpha^\vee)\not\equiv 0$ mod $p$. In other words,
$\lambda$ does not lie on any reflecting hyperplane for the dot action of the affine $p$-Weyl group $W_p$
of $G$ on ${\mathbb E}:=X(T)\otimes_{\mathbb Z}{\mathbb R}$. Let $\Xreg$ denote the set of all regular dominant weights. Observe that $\Xreg\not=\emptyset\iff p\geq h$. A (necessarily dominant) weight $\lambda$ is called $p$-restricted
(or just ``restricted") provided, for any simple root $\alpha$, $0\leq (\lambda,\alpha^\vee)<p$. Let $X_1(T)$
be the set of restricted weights.

If $\lambda,\mu\in X(T)$, define  $\lambda\leq \mu$ provided $\mu-\lambda$ is a sum of positive roots.
In this way, any subset $\Gamma$ of $X(T)_+$ can be regarded as a poset by restricting $\leq$ to $\Gamma$. If
$\Gamma\subseteq\Lambda$ are posets, write 
 $\Gamma\trianglelefteq\Lambda$ if $\Gamma$ is an ideal in $\Lambda$ in the sense that $\lambda\in
 \Lambda$ and $\lambda\leq\gamma\in\Gamma$ implies that $\lambda\in\Gamma$. We will often
 work with finite ideals in the poset $\Xreg$ or even in $X(T)_+$. For example, 
let $\Jan=\{\lambda\in
X(T)_+\,|\,(\lambda+\rho,\alpha_0^\vee)\leq p(p-h+2)\}$ be the Jantzen region. (Here $\alpha_0$ is the maximal
short root in the root system $R$ of $T$ in $G$.)

 Let $F:G\to G$ be the Frobenius morphism of $G$, defined by its ${\mathbb F}_p$-structure. For a positive integer $r$ and a rational $G$-module $V$,  $V^{[r]}$ denotes the pull-back of $V$ through $F^r$.  
Given $\lambda=\lambda_0+p\lambda_1\in X(T)_+$ with $\lambda_0\in X_1(T)$, define
rational $G$-modules
\begin{equation}\label{pmodule}\Delta^p(\lambda):=L(\lambda_0)\otimes\Delta(\lambda_1)^{[1]}\,\,\text{\rm and}\,\,\nabla_p(\lambda):=L(\lambda_0)\otimes\nabla(\lambda_1)^{[1]}.\end{equation}
Both $\Delta^p(\lambda)$ and $\nabla_p(\lambda)$ are indecomposable $G$-modules. Furthermore, the universal mapping
properties of $\Delta(\lambda)$ and $\nabla(\lambda)$ \cite[p. 183]{JanB} easily imply that $\nabla_p(\lambda)$ is a $G$-submodule of $\nabla(\lambda)$
while $\Delta^p(\lambda)$ is a $G$-homomorphic image of $\Delta(\lambda)$. In particular, $\Delta^p(\lambda)$ (resp., $\nabla_p(\lambda)$) has head (resp., socle) isomorphic to $L(\lambda)$.

Let $G_1T$--mod (resp., $\Gmod$) be the category of finite dimensional rational $G_1T$-modules (resp., $G$-modules). For $\lambda_0\in X_1(T)$, let $\whQ_1(\mu_0)$ be the injective envelope in the category $G_1T$--mod of the irreducible
$G_1T$-module $\widehat L_1(\lambda_0)$ of highest weight $\lambda_0$. When $p\geq 2h-2$, the $G_1T$-module structure on $\whQ_1(\lambda_0)$ extends uniquely to a $G$-module structure \cite[Ch. 11]{JanB}. 
It is also the projective cover of $L(\lambda_0)$ in the category of rational $G$-modules consisting of modules having composition factors $L(\gamma)$ with
$\gamma\leq\lambda_0':=2(p-1)\rho+w_0\lambda_0$. 

We will usually assume that $p\geq 2h-2$, so that, for $\lambda_0\in X_1(T)$, $\whQ_1(\lambda_0)$
can be uniquely regarded as a rational $G$-module, which will be denoted by $Q^\sharp(\lambda_0)$.
For $\lambda=\lambda_0+p\lambda_1$, $\lambda_1\in X(T)_+$, define
indecomposable rational $G$-modules by putting
\begin{equation}\label{generalizedQsandPs}
\begin{cases}
Q^\sharp(\lambda):=Q^\sharp(\lambda_0)\otimes\nabla(\lambda_1)^{[1]}\\
P^\sharp(\lambda):=Q^\sharp(\lambda_0)\otimes\Delta(\lambda_1)^{[1]}.
\end{cases}
\end{equation}
Of course, both the restrictions $Q^\sharp(\lambda)|_{G_1T}$ and $P^\sharp(\lambda)|_{G_1T}$ are injective and projective (but not
indecomposable, unless $\lambda_1=0$). 

The category $\Gmod$ of finite dimensional rational $G$-modules has a natural contravariant duality ${\mathfrak d}:\Gmod\to\Gmod$ fixing the irreducible $G$-modules; see \cite{CPS0}. The duality $\mathfrak d$
stabilizes each $Q^\sharp(\lambda_0)$, and so ${\mathfrak d}Q^\sharp(\lambda)\cong P^\sharp(\lambda)$.

 \begin{prop}\label{Extvanishcriterion} Assume that $p\geq 2h-2$. If
  $V\in\Gmod$ has a $\Delta^p$-filtration (resp., a $\nabla_p$-filtration), then 
  $\Ext^1_G(V,Q^\sharp(\mu))=0$ (resp., $\Ext^1_G(P^\sharp(\mu),V)=0$), for all $\mu\in X(T)_+$.
\end{prop}

\begin{proof}  
 We prove the $\Delta^p$-assertion, leaving the $\nabla_p$-case to the reader. First, the Hochschild-Serre spectral sequence gives, for $\mu\in X(T)_+$, an exact sequence
\begin{equation}\label{exactseq}
0\to H^1(G/G_1,\Hom_{G_1}(V,Q^\sharp(\mu))\to\Ext^1_G(V,Q^\sharp(\mu))\to\Ext^1_{G_1}(V,Q^\sharp(\mu))^{G/G_1}=0.\end{equation}
(The last term is zero because $Q^\sharp(\mu)=Q(\mu_0)\otimes\nabla(\mu_1)^{[1]}$ is an injective $G_1$-module.) Suppose that $V$ has a $\Delta^p$-filtration. To show that
$\Ext^1_G(V,Q^\sharp(\mu))=0$, for all $\mu\in X(T)_+$, we can assume that $V=\Delta^p(\lambda)$ for $\lambda=\lambda_0+p\lambda_1\in X(T)_+$. By (\ref{exactseq})
and \cite[(4.0.2)]{PS9},
$$\begin{aligned}\Ext^1_G(\Delta^p(\lambda),Q^\sharp(\mu)) &\cong \Ext^1_{G/G_1}(\Delta(\lambda_1)^{[1]}, \Hom_{G_1}(L(\lambda_0),Q(\mu_0))\otimes\nabla(\mu_1)^{[1]}) \\
&\cong \begin{cases}\Ext^1_G(\Delta(\lambda_1),\nabla(\mu_1))=0,\quad \lambda_0=\mu_0;\\
0,\quad \lambda_0\not=\mu_0\end{cases}\end{aligned}$$
as required. (Here we have used \cite[II.4.13 Prop.]{JanB} and the isomorphism $\Delta(\lambda_1)^*
\cong\nabla(\lambda_1^\star)$.)
\end{proof}
\begin{lem}\label{extensionlemma} Assume that $p\geq 2h-2$. For $\lambda,\mu\in X(T)_+$, we have
$$\nu\leq\mu\implies \Ext^1_G(P^\sharp(\mu),L(\nu))=0 =\Ext^1_G(L(\nu),Q^\sharp(\mu)).$$\end{lem}

\begin{proof} Suppose that $\Ext^1_G(L(\nu),Q^\sharp(\mu))\not=0$. Write $\nu=\nu_0+p\nu_1$ with $\nu_0\in X_1(T)$ and $\nu_1\in X(T)_+$. Since $\soc_{G_1}Q(\mu_0)=L(\mu_0)$, the exact sequence (\ref{exactseq})
with $V=L(\nu)$
implies that $\nu_0=\mu_0$ and $\Ext^1_G(L(\nu_1),\nabla(\mu_1))\not=0$. Therefore, by \cite[(4.0.1)]{PS9}, $\nu_1>\mu_1$, so $\nu>\mu$. This proves
the lemma for $Q^\sharp(\mu)$. A similar argument establishes it for $P^\sharp(\mu)$.\end{proof} 

\begin{prop}\label{prop2.4} Assume that $p\geq 2h-2$. Let $\lambda=\lambda_0+p\lambda_1\in X(T)_+$, where $\lambda_0\in X_1(T)$ and $\lambda_1\in X(T)_+$.

(a) $P^\sharp(\lambda)$ has a $\Delta$-filtration and $Q^\sharp(\lambda)$ has a $\nabla$-filtration.

(b) $P^\sharp(\lambda)$ and $\Delta(\lambda)$ have $G_1$-heads isomorphic to $\Delta^p(\lambda)$ as rational
$G$-modules. Similarly, $Q^\sharp(\lambda)$ and $\nabla(\lambda)$ have $G_1$-socles
isomorphic to $\nabla_p(\lambda)$ as rational $G$-modules.
\end{prop}

\begin{proof} We prove (a) and (b) only for $Q^\sharp(\lambda)$ (and $\nabla(\lambda)$), leaving the 
other cases to the reader.

Proof of (a): By its construction, the $G$-module $Q^\sharp(\lambda_0)$
is a direct summand of $\nabla(\xi)\otimes\St$, for some $\xi\in X(T)_+$; see \cite[Ch. 11]{JanB}. Therefore,
using \cite[II.3.19]{JanB}, $Q^\sharp(\lambda)$ is a direct summand of 
$$\nabla(\xi)\otimes\St\otimes\nabla(\lambda_1)^{[1]}\cong
 \nabla(\xi)\otimes\nabla((p-1)\rho\otimes p\lambda_1).$$
 Since a tensor product of modules with a $\nabla$-filtration have a $\nabla$-filtration \cite[II.4.21]{JanB},  $Q^\sharp(\lambda)$ has a $\nabla$-filtration.
 
 Proof of (b): We have, writing $\soc$ for $\soc_{G_1}$,
 $$\soc Q^\sharp(\lambda)=\soc(Q(\lambda_0)\otimes \nabla(\lambda_1)^{[1]}
 \cong \soc(Q(\lambda_0))\otimes\nabla(\lambda_1)^{[1]}\cong L(\lambda_0)\otimes\nabla(\lambda_1)^{[1]}
 \cong \nabla_p(\lambda).$$

 Lemma \ref{extensionlemma} easily implies that the  inclusion $L(\lambda)\hookrightarrow Q^\sharp(\lambda)$ 
 of rational $G$-modules extends to a morphism $\nabla(\lambda)\hookrightarrow Q^\sharp(\lambda)$.
Therefore, $\soc_{G_1}\nabla(\lambda)\subseteq\soc_{G_1} Q^\sharp(\lambda)=\nabla_p(\lambda)$. On the other hand,
$\nabla_p(\lambda)\hookrightarrow \nabla(\lambda)$, so $\nabla_p(\lambda)\subseteq\soc\nabla(\lambda)$. Thus, $\nabla_p(\lambda)=\soc_{G_1}\nabla(\lambda)$.
\end{proof}

In addition to the $G$-modules $\Delta^p(\lambda)$ and $\nabla_p(\lambda)$, $\lambda\in X(T)_+$,
another family of modules, denoted $\rDelta(\lambda)$ and $\rnabla(\lambda)$, will play an important
role in this paper. These modules, first studied by Lin \cite{Lin}  (with the $\rDelta(\lambda)$ making a brief appearance in 
 Lusztig \cite{L1}), involve the theory of quantum groups at roots of unity.

Let $\zeta\in\mathbb C$ be a primitive  $p$th root of 1. Let $\sA={\mathbb Z}[v,v^{-1}]_{\mathfrak n}$, where ${\mathfrak n}=(v-1,p)$. We regard $\sA$ as a subring of the function field ${\mathbb Q}(v)$. Define $\sO':=\sA/(\phi_p)$, where $\phi_p=1+v+\cdots
+v^{p-1}\in{\mathfrak n}$. Thus, $\sO'$ is a DVR with residue field ${\mathbb F}_p$. It will be convenient to
enlarge $\sO'$ to be a complete DVR $\sO$ with residue field $k:=\overline{{\mathbb F}_p}$ and fraction field $K$.
Thus, $(K,\sO,k)$ will from now on be a $p$-modular system with
residue field the algebraically closed field $k$ as defined above. The natural mapping $\sA\to\sO$ takes
$v$ to $\zeta$.

Let $\widetilde U'_\zeta$ be the (Lusztig) $\sA$-form
of the quantum enveloping algebra ${\mathbb U}_v$ associated to the Cartan matrix of the root system $R$ over the function
field ${\mathbb Q}(v)$.  
 Let
$$\widetilde U_\zeta=\sO\otimes_\sA U'_\zeta/\langle K_1^p-1,\cdots, K_n^p-1\rangle.$$
Finally, set $U_\zeta=K\otimes_{\sO}\widetilde U_\zeta$, so that $\widetilde U_\zeta$ is an integral
$\sO$-form of the quantum enveloping algebra $U_\zeta$ at a $p$th root of unity. 
Put $\overline U_\zeta=\widetilde
U_\zeta/\pi\widetilde U_\zeta$, and let $I$ be the ideal in
$\overline U_\zeta$ generated by the images of the elements $K_i-1$,
$1\leq i\leq n$. By \cite[(8.15)]{L1},
\begin{equation}\label{hyperalgebra}
\overline U_\zeta/I\cong \Dist(G),\end{equation}
 the
distribution algebra of $G$ over $k$. 
 
The category $U_\zeta$--mod of
finite dimensional, integrable, type 1 $U_\zeta$-modules is a highest weight category (in the sense
of \cite{CPS-1}) with
irreducible (resp. standard, costandard) modules $L_\zeta(\lambda)$
(resp., $\Delta_\zeta(\lambda)$, $\nabla_\zeta(\lambda)$),
$\lambda\in X(T)_+$. For $\mu\in X(T)_+$, $\text{\rm
ch}\,\Delta_\zeta(\mu)=\text{\rm ch}\,\nabla_\zeta(\mu)=\chi(\mu)$ (Weyl's character formula). In the
sequel, it will be sometimes convenient to denote $L_\zeta(\lambda)$, $\Delta_\zeta(\lambda)$,
and $\nabla_\zeta(\lambda)$ by $L_K(\lambda)$, $\Delta_K(\lambda)$, and $\nabla_K(\lambda)$, 
respectively.

\begin{center} {\bf The Lusztig character formula (LCF)}\end{center}

\medskip In the following three paragraphs, we discuss the Lusztig character formula (needed in Lemma \ref{LCF}(c) and \S4 below). The first paragraph considers the case of the algebraic group $G$, the second considers
the case of the quantum enveloping algebra $U_\zeta$, while the third paragraph combines the two cases.

Assume that $p\geq h$, and let $$C^-:=\{x\in {\mathbb E}:=X(T)\otimes_{\mathbb Z}{\mathbb R}\,|\,-p<(\lambda+\rho,\alpha^\vee)<0,\quad\forall\alpha\in R_+\}$$ be the unique alcove for the affine Weyl group $W_p$ containing $-2\rho$.  Viewing $W_p$ as a Coxeter group with fundamental reflections in
the walls of $C^-$, for $y,x\in W_p$, let $P_{y,x}$ be the associated Kazhdan-Lusztig polynomial (in a variable $q=t^2$). Given a dominant weight $\lambda$, write $\lambda=x\cdot\lambda^-$,
for a unique $\lambda^-\in  \overline{C^-}$ and a unique $x\in W_p$ of minimal length. The
Lusztig character formula is the formal expression 
$$\chi_{\text{KL}}(\lambda):=\sum_y(-1)^{\ell(x)-\ell(y)}P_{y,x}(1)\ch \Delta(y\cdot\lambda^-),$$
where the sum is over all $y\in W_p$ satisfying $y\leq x$ and $y\cdot\lambda\in X(T)_+$. If
$\ch\,L(\lambda)=\chi_{\text{KL}}(\lambda)$, we say the LCF holds for $L(\lambda)$  (or for the dominant
weight $\lambda$). We will say that the LCF holds for $G$ if it holds for every restricted dominant weight
$\lambda$. Equivalently, using a standard Jantzen translation functor argument, the LCF holds for $G$
if and only if it holds for every regular restricted dominant weight).\footnote{The original Lusztig modular
conjecture \cite{L} posits that the LCF formula holds for all dominant weights $\lambda$ belonging to the Jantzen
region $\Jan$. If the LCF holds for $G$ and $p\geq 2h-3$, then the restricted weights are contained in the Jantzen
region and, using Kato \cite[p. 128]{Kato}, the Lusztig conjecture holds. In this paper, we will usually assume
that $p\geq 2h-2>2h-3$.} By \cite{AJS}, 
 the LCF
holds $G$ if $p\gg 0$ (depending on $R$, but specific bounds not given). More recently, explicit (and large!) bounds on $p$ insuring 
 the validity of the LCF for $G$ are available in Fiebig \cite[Cor. 9.9 and p. 135]{Fiebig}.  The original Lusztig conjecture required only a bound $p\geq h$.  However, a recent announcement of counterexamples 
 by Williamson \cite{W13} indicates this bound must be increased.\footnote{Replacing $h$ by the order of
 the Weyl group would not appear to contradict any of Williamson's proposed counterexamples. We note that the orders of
 the Weyl groups are still much smaller than the bounds obtained by Fiebig, which are huge.} 

Now consider the case of $U_\zeta$-mod and assume that $p>h$. For any $\lambda\in X(T)_+$,  then the irreducible, finite dimensional $U_\zeta$-module $L_\zeta(\lambda)$ has character satisfying $\ch\,L_\zeta(\lambda)=\chi_{\text{KL}}(\lambda)$. Thus, for $p>h$, the  ``LCF holds
for $U_\zeta$-mod."
  See \cite[\S7]{T} for a detailed discussion and
further references.

\medskip

As discussed in \cite{CPS7} (and earlier in \cite{Lin}), given $\lambda\in X(T)_+$,  there exist admissible lattices
$\widetilde\Delta_\zeta(\lambda)$ and
$\widetilde\nabla_\zeta(\lambda)$ for $\Delta_\zeta(\lambda)$ and
$\nabla_\zeta(\lambda)$, respectively, so that
$\widetilde\Delta_\zeta(\lambda)/\pi\widetilde\Delta_\zeta(\lambda)\cong
\Delta(\lambda)$ and
$\widetilde\nabla_\zeta(\lambda)/\pi\widetilde\nabla_\zeta(\lambda)\cong
\nabla(\lambda)$. The lattice $\widetilde\Delta_\zeta(\lambda)$
is generated as a $\widetilde U_\zeta$-module by a highest weight vector in $\Delta_\zeta(\lambda)$.
In the sequel, $\wDelta_\zeta(\lambda)$ is denoted simply by $\wDelta(\lambda)$ and 
$\Delta_\zeta(\lambda)$ by $\Delta_K(\lambda)$.

Given $\lambda\in X(T)_+$,  fix a highest weight vector $v^+\in
L_\zeta(\lambda)$. Then there is a unique admissible lattice
$ \wrDelta(\lambda)$ (resp., $\wrnabla(\lambda)$) of $L_\zeta(\lambda)$ which is
minimal (resp., maximal) with respect to all admissible lattices
$\widetilde L$ such that $\widetilde L\cap
L_\zeta(\lambda)_\lambda=\sO v^+$. For example, put
$\wrDelta(\lambda)=\widetilde
U_\zeta\cdot v^+$. By abuse of notation, $\wrDelta(\lambda)$
 (resp., $\wrnabla(\lambda)$) is called the minimal (resp., maximal) lattice of
$L_\zeta(\lambda)$. Any two ``minimal" (resp., ``maximal") lattices
are isomorphic as $\widetilde U_\zeta$-modules.

For $\lambda\in X(T)_+$, put
\begin{equation}\label{reduced}\rDelta(\lambda):=\wrDelta(\lambda)/\pi\wrDelta(\lambda)\,\,{\text{\rm and}}\,\,\rnabla(\lambda) :=
\wrnabla(\lambda)/\pi\wrnabla(\lambda).\end{equation}
 As noted above the LCF holds for $U_\zeta$--mod if $p>h$, giving
$$\ch\,\rDelta(\lambda)=\ch\,\rnabla(\lambda)=\chi_{\text{KL}}(\lambda).$$

\begin{prop}\label{Lin} (a) (\cite[Thm. 2.7]{Lin} or \cite[Prop. 1.78]{CPS7}) Assume that $\mu=\mu_0+p\mu_1
\in X(T)_+$ for $\mu_0\in X_1(T)$ and $\mu_1\in X(T)_+$.
Then $\rDelta(\mu)\cong\rDelta(\mu_0)\otimes\Delta(\mu_1)^{[1]}$ and $\rnabla(\mu)
\cong\rnabla(\mu_0)\otimes\nabla(\mu_1)^{[1]}$. 

(b) Assume that $p\geq 2h-2$. Then, given any $\mu\in X(T)_+$, $\rDelta(\mu)$ (resp., $\rnabla(\mu)$)
has a $\Delta^p$-filtration (resp., $\nabla_p$-filtration). In particular, $\Ext^1_G(V,Q^\sharp(\lambda))=0$
for all $\lambda\in X(T)_+$ and any $V\in\Gmod$ having a $\rDelta$-filtration.
\end{prop}
\begin{proof} It suffices to prove statement (b).  By (a) and its notation, $\rDelta(\mu)\cong\rDelta(\mu_0)
\otimes\Delta(\mu)^{[1]}$. Any tensor product $\Delta(\gamma)^{[1]}\otimes\Delta(\delta)^{[1]}\cong
(\Delta(\gamma)\otimes\Delta(\delta))^{[1]}$ has a filtration with sections $\Delta(\omega)^{[1]}$. 
Thus, it is enough to show that $\rDelta(\mu_0)$ has a $\Delta^p$-filtration. Because $p\geq 2h-2$, any restricted dominant weight belongs to the Janzten
region, which consists of all $\lambda\in X(T)_+$ satisfying $(\lambda+\rho,\alpha_0^\vee)\leq p(p-h+2)$. 
If $L(\tau)$ is a composition factor of $\rDelta(\mu_0)$, then $\tau\leq\mu_0$, so $\tau$ also lies in the Jantzen
region. Therefore, writing $\tau=\tau_0+p\tau_1$ ($\tau_0\in X_1(T)$, $\tau_1\in X(T)_+$), it follows that
$\tau_1$ lies in the closure of the bottom $p$-alcove $C$ (see \cite[II, \S6.2]{JanB}).  Hence, $L(\tau_1)\cong\Delta(\tau_1)$,
so that $L(\tau)\cong L(\tau_0)\otimes L(\tau_1)^{[1]}\cong L(\tau_0)\otimes\Delta(\tau_1)^{[1]}\cong
\Delta^p(\tau)$, as required. The final assertion follows from Lemma \ref{Extvanishcriterion}(b). \end{proof}

Part (c) in the next result is an immediate consequence of part (a) of the above proposition and the
validity of the LCF for $U_\zeta$-mod as long as $p>h$. In fact,  if the LCF holds for $G$,  then $\rDelta(\lambda)=L(\lambda)$, $\lambda$ restricted.

\begin{cor} \label{LCF} (a) Let $\mu\in X(T)_+$, and let $\wM$ be a $\wU_\zeta$-lattice, and set
 $M:=(\wM)_k$ as usual.  Assume $\wM_K=L_\zeta(\mu)$ and that $L(\mu)$ is the $G$-socle of
 $M$. Then $M\cong\rnabla(\mu)$. 
 
 (b) Similarly, for $\mu\in X(T)_+$, and assume $\wM$ is a $\wU_\zeta$-lattice such that $\wM_K\cong L_\zeta(\mu)$ and such that the head of $M$ is $L(\mu)$ as a $G$-module. Then $M\cong\rDelta(\mu)$.

(c) Assume that $p>h$ is a prime. The LCF holds
for $G$ for all irreducible modules having restricted highest weights if and only if $\rDelta(\mu)=\Delta^p(\mu)$, for all $\mu\in X(T)_+$. \end{cor}

\begin{proof} As noted above, it is sufficient to address (a) and (b). We prove (b). By hypothesis, 
$M$ has irreducible head $L(\mu)$. Thus, $M$ is a cyclic module generated by a $\mu$-weight vector. Thus, $\wM$ is also cyclic, generated by a $\mu$-weight vector $v$. Since $\wM_K\cong L_\zeta(\mu)$, $\wM_\mu$ is an $\sO$-lattice of rank 1. Thus, $\wM$ is the unique, up to isomorphism, cyclic $\wU_\zeta$-lattice.
Therefore, $\wM\cong\wrDelta(\mu)$, and so $M\cong\rDelta(\mu)$, proving (b). Taking duals, proves
(a). \end{proof}

Now let $u_\zeta$ be the small quantum group associated to $U_\zeta$ and the Frobenius
map $F:U_\zeta\to U({\mathfrak g}_{\mathbb C})$ (the  universal
enveloping algebra of the complex simple Lie algebra ${\mathfrak g}_{\mathbb C}$ with root system $R$). Thus, $u_\zeta$ is a
normal subalgebra of $U_\zeta$ (as well as a Hopf subalgebra) of dimension $p^{\dim{\mathfrak g}_{\mathbb C}}$, and $U_\zeta//u_\zeta
\cong U({\mathfrak g})$. The algebra $u_\zeta$ also admits an $\sO$-form
$\widetilde u_\zeta$ whose image in $\Dist(G)$ is isomorphic to the restricted enveloping algebra $u=u({\mathfrak g})$. It is easy to see that, given
any $M\in U_\zeta$--mod, the $U_\zeta$-socle of $M$ agrees with the $u_\zeta$-socle of $M|_{u_\zeta}$.
Given $V\in U({\mathfrak g}_{\mathbb C})$--mod, let $V^{[1]}\in U_\zeta$--mod denote the pull-back of $M$ through the
Frobenius map $F$. (Although this notation conflicts with similar notation for $G$, it should not
cause confusion.)

Suppose $\Gamma\subset\Xreg$ is a finite non-empty ideal in the poset $\Xreg$, 
and let $U_\zeta{\text{\rm --mod}}[\Gamma]$ be the full
subcategory of $U_\zeta$--mod consisting of those modules with composition factors $L_\zeta(\lambda)$, $\lambda\in\Gamma$. Then $\Uzmod[\Gamma]$ is a
highest weight category with weight poset $\Gamma$, standard (resp., costandard, irreducible) modules $\Delta_\zeta(\lambda)$ (resp., $\nabla_\zeta(\lambda)$,
$L_\zeta(\lambda)$)
for $\lambda\in\Gamma$. If $J$ is the annihilator of $\Uzmod[\Gamma]$, then $\Uzmod[\Gamma]\cong U_\zeta/J$--mod and the algebra $A_{\zeta,\Gamma}=U_\zeta/J$
is a quasi-hereditary algebra with weight poset $\Gamma$. Further, putting $\wA_{\zeta,\Gamma}$ equal to the image of $\wU_\zeta$ in $A_{\zeta,\Gamma}$,
$A_\Gamma:=\wA_{\zeta,\Gamma}/\pi\wA_{\zeta,\Gamma}$ satisfies $A_\Gamma{\text{\rm --mod}}\cong\Gmod[\Gamma]$, the full subcategory of the $G$--mod consisting of finite dimensional rational $G$-modules
with composition factors $L(\gamma)$, $\gamma\in\Gamma$.

The image $\wA_{\zeta,\Gamma}$  of $\wU_\zeta$ in $A_{\zeta,\Gamma}$ is an $\sO$-subalgebra  $A_{\zeta,\Gamma}$, necessarily an
$\sO$-lattice. Furthermore, $\overline{\wA_{\zeta,\Gamma}}=A_{\Gamma}$, using (\ref{hyperalgebra}). In addition, all projective modules for $A_{\Gamma}$ lift uniquely to $\wA_{\zeta,\Gamma}$-lattices which are projective; see \cite{DS}.

It will generally be necessary to enlarge the poset $\Gamma$. Thus, let $\Lambda$ be a subposet in $\Xreg$  containing $\Gamma$ as an ideal, and with the property that 
\begin{equation}\label{w0}2(p-1)\rho+w_0\gamma_1 +p\gamma_1\in\Lambda,\quad\forall\gamma\in\Gamma,\end{equation}
where $\gamma=\gamma_0+p\gamma_1$, $\gamma_0\in X_1(T),\gamma_1
\in X(T)_+$. This implies that all the highest weights of the $G$-composition factors of $P^\sharp(\gamma)$, for
$\gamma\in\Gamma$, belong to $\Lambda$. The module $P^\sharp(\gamma)$ is defined as a $G$-module
when $p\geq 2h-2$, {\it a condition we now assume}; see (\ref{generalizedQsandPs}), where $\gamma
=\lambda$. The $G$-module $Q^\sharp(\gamma_0)$ is a projective indecomposable module in $p$-bounded subcategory of $\Gmod$ defined in 
\cite[II.11.11]{JanB}. Thus, for $\Gamma'$ the poset ideal associated to the regular weights in the $p$-bounded category, $Q^\sharp(\gamma_0)$ is a projective $A_{\Gamma'}$-module. As noted above, it lifts uniquely
to a projective $\wA_{\Gamma'}$-lattice $\widetilde Q^\sharp(\gamma_0)$. Now regard, $\widetilde Q^\sharp(\gamma_0)$
as a $\widetilde U_\zeta$-module, and define 
\begin{equation}\label{wP}\wP^\sharp(\gamma):=\wQ^\sharp(\gamma_0)\otimes\wDelta(\gamma_1)^{[1]}.\end{equation}
This is a $\wU_\zeta$-module, and also $\wA_\Lambda$-lattice, because of our requirements on $\Lambda$. 
Similarly, we can define a $\wA_\Lambda$-lattice (or a $\wU_{\zeta,\Lambda}$-module)
\begin{equation}\label{wQ}
\wQ^\sharp(\gamma):=\wQ^\sharp(\gamma_0)\otimes\widetilde\nabla(\gamma_1)^{[1]}\end{equation}

As a consequence of this discussion,  the following result holds.

\begin{prop}\label{prop2.5} Assume that $p\geq 2h-2$ is an odd prime. Let $\Gamma$ be a finite, non-empty ideal
in $\Xreg$ and let $\Lambda$ be as above. 

(a) The modules defined in (\ref{wP}) and (\ref{wQ}) satisfy
$$\begin{cases} \overline{\wP^\sharp(\gamma)}\cong P^\sharp(\gamma);\\
\overline{\wQ^\sharp(\gamma)  }\cong Q^\sharp(\gamma)\end{cases}$$
for all $\gamma\in\Gamma$, where $Q^\sharp(\gamma)$ and $P^\sharp(\gamma)$ are defined
as in (\ref{generalizedQsandPs}).

(b) For $\gamma\in\Gamma$, $\wP^\sharp(\gamma)$ has a $\wDelta$-filtration and $\wQ^\sharp(\gamma)$
has a $\widetilde\nabla$-filtration.\end{prop}

\begin{proof} Part (a) follows from the construction of $\widetilde P^\sharp(\gamma)$ and $\wQ^\sharp(\gamma)$. To see (b), first observe that by Proposition \ref{prop2.4} that $P^\sharp(\gamma)$ has a
$\Delta$-filtration.  Therefore, by (the dual version of) \cite[II.4.16(b)]{JanB}, $\Ext^1_{A_\Lambda}(P^\sharp(\lambda),\nabla(\mu))=
 \Ext^1_G(P^\sharp(\lambda),\nabla(\mu))=0$, for all $\mu\in\Lambda$. By a standard base change result \cite[Lem. 1.5.2(c)]{CPSMem},
 this implies that $\Ext^1_{\wA_\Lambda}(\wP^\sharp(\lambda),\wnabla(\mu))=0$, for all $\mu\in\Lambda$.
 Therefore, Proposition \ref{filtrationprop} in Appendix I (\S6) implies that $\wP^\sharp(\lambda)$ has a $\wDelta$-filtration. A dual argument works to show that each $\wQ^\sharp(\gamma)$ has a $\wnabla$-filtration. 
\end{proof}

\section{A summary of previous results; new results on tight modules}
\setcounter{equation}{0}

This section begins by summarizing various results from \cite{CPS1a}, \cite{DS}, \cite{PS10}. Some important (and new) facts concerning tight modules will be established. All this will be needed later  for the main results of the paper. 
Throughout fix an odd prime $p$ satisfying  $p\geq 2h-2$.   The size restriction on $p$ is only needed in order to make use of the graded results established in \cite{PS10}.

Let $\Lambda$ be a finite, nonempty ideal in the poset $\Xreg$ of regular weights in $X(T)_+$.    
Consider the $\sO$-algebra  $\wA_\Lambda=\wU_{\zeta,\Lambda}$ constructed in \S2. It is an $\sO$-lattice. When the poset $\Lambda$ is 
understood, denote $\wA_\Lambda$ simply by $\wA$. If $\emptyset\not=\Gamma\trianglelefteq
\Lambda$, there is a natural surjection $\wA_\Lambda\twoheadrightarrow\wA_\Gamma$. In this way, any $\wA_\Gamma$-module $\wM$ can be viewed, by inflation, as an $\wA_\Lambda$-module.
 Given two finite $\wA_\Gamma$-modules $\wM,\wN$,  inflation induces isomorphisms
\begin{equation}\label{Ext}\Ext^n_{\wA_\Gamma}(\wM,\wN)\cong
\Ext^n_{\wA_\Lambda}(\wM,\wN)\quad\forall n\geq 0.\end{equation}
These results follow from the general theory of quasi-hereditary algebras \cite{CPS1a} and its discussion
in the quantum case in \cite{DS} (the latter results are also recalled in \cite{PS10}). The algebras
$\wA_\Lambda$ are all split quasi-hereditary algebras (QHAs) in the sense of \cite{CPS1a}.

A similar result holds at the graded level. More precisely, consider the $\sO$-algebras $\gr\wA_\Lambda$ and $\gr\wA_\Gamma$. These algebras 
 are both (positively) graded QHAs with weight posets $\Lambda$ and $\Gamma$, respectively. There is a surjective homomorphism $\gr\wA_\Lambda\twoheadrightarrow\gr\wA_\Gamma$ which
induces (by restriction) an isomorphism
\begin{equation}\label{grExt}\Ext^n_{\gr\wA_\Gamma}(\wM,\wN)\cong\Ext^n_{\gr\wA_\Lambda}(\wM,\wN),\quad\forall n\geq 0.\end{equation}
for any two finite $\gr\wA_\Gamma$-modules $\wM,\wN$. In addition, these isomorphisms hold at the
level of graded $\grExt^\bullet$, when $\wM,\wN$ are graded $\gr\wA_\Gamma$-modules. Also,
if $M_K,N_K$ are $(\gr \wA_\Gamma)_K$-modules and if $M,N$ are $(\gr \wA_\Gamma)_k$-modules,
then
\begin{equation}\label{level K/k}
\begin{cases}\Ext^n_{(\gr\wA_\Gamma)_K}(M_K,N_K)\cong\Ext^n_{(\gr\wA_\Lambda)_K}(M_K,N_K);\\
\Ext^n_{(\gr\wA_\Gamma)_k}(M,N)\cong\Ext^n_{(\gr\wA_\Lambda)_k}(M,N),\end{cases}\quad\forall n\geq 0.\end{equation}
A closely related result is the isomorphism 
\begin{equation}\label{closelyrelatedresult}(\gr\wA_\Lambda)_\Gamma\cong\gr \wA_\Gamma,\end{equation}
 where $(\gr\wA_\Lambda)_\Gamma$ denotes the largest quotient algebra of $\gr\wA_\Lambda$, all of whose irreducible
modules have the form $L(\gamma)$ for $\gamma\in\Gamma$. 

All these quoted results in the previous paragraph follow from the (split)
quasi-heredity of $\gr\wA_\Lambda$ and the description of its standard modules (as the graded
modules $\gr\widetilde\Delta(\lambda)$, $\lambda\in\Lambda$) proved in \cite{PS10}; see especially Remark 3.18 and Theorem 5.3 there.

For the algebras $\wA_\Lambda$,  a key step in \cite{PS10} in proving that $\gr\wA_\Lambda$ is an integral QHA is  
showing that each $\gr\wA_\Lambda$-module $\gr\widetilde\Delta(\lambda)$, $\lambda\in\Lambda$, has  a simple head. As a consequence of this fact we also record the following result, using the
proof of \cite[Cor. 3.15]{PS10}.

\begin{thm}\label{gradedDeltas}     Let $\wN$ be a $\wA_\Lambda$-lattice which has a $\widetilde\Delta$-filtration. Then the graded $\gr\wA_\Lambda$-module $\gr\wN$ has a $\gr\widetilde\Delta$-filtration. In addition, the multiplicity of $\widetilde\Delta(\nu)$  
as a section of $\wN$ agrees with the multiplicity of $\gr\widetilde\Delta(\nu)$ as an ungraded section
of $\gr\wN$, or, equivalently, with the sum of the multiplicities of the various shifts of $\gr\widetilde\Delta(\lambda)$ as graded sections of $\gr\wN$.\end{thm}

This leads to the following corollary.

\begin{cor}\label{ZeroDegreeQHA}Let $\wA_\Lambda$ be as above  and form the graded integral QHA $\gr\wA_\Lambda$
with weight poset $\Lambda$. Then $\wB:=(\gr\wA_\Lambda)_0$ (the term in grade 0 in $\gr\wA_\Lambda$)
is an integral QHA with weight poset $\Lambda$. It has 
standard (resp., costandard) modules $\wrDelta(\lambda)$ (resp., $\wrnabla(\lambda)$), $\lambda\in\Lambda$.\end{cor} 

\begin{proof} The projective indecomposable modules for $\wB$ are just the modules $\gr\wP(\lambda)_0$, $\lambda\in\Lambda$,  where
$\wP(\lambda)$ is the projective cover of the irreducible $\wA_\Lambda$-module $L(\lambda)$.  By Theorem
\ref{gradedDeltas}, the graded module $\gr\wP(\lambda)$ has a graded filtration by modules $\gr\widetilde\Delta(\nu)$.   Since the head of $\wP(\lambda)$ is $L(\lambda)$, the top section of this filtration must
be $\gr\widetilde\Delta(\lambda)$, the only module $\gr\widetilde\Delta(\lambda)$ with head $L(\lambda)$. 
By the multiplicity assertion of Theorem \ref{gradedDeltas}, all other sections $\gr\widetilde\Delta(\nu)$  of the filtration satisfy $\nu>\lambda$. 

Now pass to grade 0 to obtain a filtration of $\gr\wP(\lambda)_0$ with top section $\gr\wDelta(\lambda)_0$
and lower sections $\gr\wDelta(\nu)_0$ for $\nu>\lambda$. Clearly, any composition factor of $\gr\wDelta(\nu)_0$ has the form $L(\tau)$, $\tau\leq \lambda$, with $L(\lambda)$ occurring just once, in the head.
It follows that $\wB$ is an integral  QHA, as required.

Finally, $\gr\wDelta(\lambda)_0\cong\wDelta(\lambda)/\wrad\Delta(\lambda)$ is generated by a $\lambda$-weight vector as a $\wU_\zeta$-module, so gives (the unique, up to isomorphism) minimal $\wU_\zeta$-lattice
in the irreducible module $((\gr\wDelta(\lambda)_0)_K\cong L_\zeta(\lambda)$. That is,
$\gr\wDelta(\lambda)_0\cong\wrDelta(\lambda)$.
 \end{proof}
 
There are natural dualities $\wdelta_K$ on $U_\zeta$--mod and $\delta=\wdelta_k$ on $G$-mod which
fix irreducible modules. Both are compatible by base change with a duality $\wdelta$ on the category of
$\wU_\zeta$- or $\wA$-lattices. 
   It takes $\wP^\sharp(\lambda)$ to
$\wQ^\sharp(\lambda)$, and also satisfies ${\widetilde{\mathfrak d}}\wDelta(\lambda)\cong\wnabla(\lambda)$ and ${\mathfrak d}(\gr\wP^\sharp(\lambda))\cong\gr^\diamond\wQ^\sharp(\lambda)$, for all $\lambda\in\Lambda$. 

 \begin{cor}\label{cor3.3} Let $\Gamma\trianglelefteq\Lambda$ satisfy condition (\ref{w0}) for all
 $\gamma\in\Gamma$, so that, if $\gamma\in\Gamma$,  then  $\wP^\sharp(\gamma)$ and $\wQ^\sharp(\gamma)$ ) belong to the category $\wA_\Lambda$-mod. Then, for $\gamma\in\Gamma$,  $\gr\wP^\sharp(\gamma)$ (resp., $\gr^\diamond \wQ^\sharp(\gamma)$) has,
 as a graded $\gr\wA_\Lambda$-module, a $\gr\wDelta$-filtration (resp., $\gr\wnabla$-filtration). Thus,
 there is a filtration of $\gr\wP^\sharp(\gamma)$ (resp., $\gr^\diamond \wQ^\sharp(\gamma)$) as a graded $\gr\wA_\Lambda$-module with sections of the
 form $\gr\wDelta(\mu)(r)$ (resp., $\gr\wnabla(\mu)(-r)$),  $r\in\mathbb N$, $\mu\in\Lambda$. \end{cor}
 
 \begin{proof} By remarks before the corollary, it suffices to consider the case of $\wP^\sharp(\gamma)$. By Proposition \ref{prop2.5}(b),
 $\wP^\sharp(\gamma)$ has a $\wDelta$-filtration.
 Now apply Theorem \ref{gradedDeltas} to complete the proof.\end{proof}

\medskip
\begin{center}{\bf The subalgebra $\bold\wfa$ and the modules $\bold\wgr M$}\end{center}
\smallskip
  We make the following assumptions throughout the remainder of this section.

\begin{assumptions}\label{assumptions} The prime $p$ is assumed to be odd and $\geq 2h-2$. $\Lambda$ is a finite, non-empty ideal in $\Xreg$, which contains the set $\Xreg\cap X_1(T)$ of regular restricted weights. In addition, if $\lambda\in \Xreg\cap X_1(T)$,   then $2(p-1)\rho+w_0\lambda\in\Lambda$. Let $\wA=\wA_\Lambda$.  \end{assumptions}

 Let $\wfa$ be the image of the (integral) small quantum group $\widetilde u_\zeta$ in $\wA_\Lambda$. Our assumptions
on $\Lambda$ imply that $\wfa_K$ is isomorphic to the projection $u'_{\zeta,K}$ of $u_{\zeta,K}$ on its regular blocks. For more details/discussion, see \cite[\S5]{PS10}. Also, 
\begin{equation}\label{radicals}(\rad^n\wfa_K)\wA_K=
  \rad^n\wA_K,\quad\forall n\in{\mathbb N}, \end{equation}
by \cite[Lem. 8.1, Lem. 8.3]{PS9}. In addition, we have $\wA_K(\rad^n\wfa_K)=\rad^n\wA_K$,
  for $n\geq 0$; see \cite[Theorem 6,1(b)]{PS10} (the condition on ``fatness" there is not necessary, since
  radicals are preserved under surjective homomorphisms).
  
  Also, $\fa:=\wfa_k$ is isomorphic (through the projection $(\widetilde u_\zeta)_k\to\wfa_k=\fa$) to the (direct) sum of the regular blocks in the universal enveloping algebra $u$ of the Lie
  algebra of $G$. See \cite[fn. 17]{PS10} for a discussion.

  In the following, let $\wrad^n\wfa:=\wfa\cap\rad^n\wfa_K$ (as discussed indirectly in item (6) in the notation of \S2), and put $\gr\wfa:=\bigoplus_{n\geq 0}\wrad^n\wfa/\wrad^{n+1}\wfa$. Consider the decreasing filtration 
 \begin{equation}\label{decreasing}\wF^\bullet: \wfa\supseteq\wrad\wfa\supseteq\wrad^2\wfa\supseteq\cdots
 \end{equation}
 of $\wfa$. Thus, $\wF^n=\widetilde \rad^n\wfa$. Given an $\wfa$-module $\wM$ (which could be an
 $\fa$-module $M$\footnote{In fact, we mostly use $\wgr\wM$ when $\wM=M$ is an $A$-module, where $A=\overline{\wA}$. For $\wfa$ lattices $\wM$, $\wgr\wM$ is poorly behaved because of the possibility of torsion, unless $\wM$ is
 a tight lattice, in which case $\wgr\wM=\gr\wM$.}), we define a
 graded $\gr\wfa$-module $\wgr \wM$ by setting 
 \begin{equation}\label{wgr} \wgr\wM:=\bigoplus_{n\geq 0}\wF^n\wM/\wF^{n+1}\wM=\bigoplus_{n\geq 0}
 (\wrad^n\wfa)\wM/(\wrad^{n+1}\wfa)\wM.\end{equation}
 In particular, $\wgr\wfa=\gr\wfa$, and it follows that $\wgr\wM$ is a graded module for $\gr\wfa$.

 It is useful to observe the following lemma. 

\begin{lem}\label{pretight} Let $\wM$ be an $\wfa$-lattice. There is a natural map 
$\wgr\wM\to\gr\wM$ of $\gr\wfa$-modules. The lattice $\wM$ is $\wfa$-tight if and
only if this map is surjective. Moreover, if the map is surjective, then it is an isomorphism
of $\gr\wfa$-modules, and there is a physical equality
$\wgr\wM=\gr\wM$.  
   \end{lem}
   
\begin{proof}  There is an evident, natural inclusion
\begin{equation}\label{tightn}(\wrad^n\wfa)\wM\hookrightarrow\rad^n\wM_K\cap\wM\end{equation} 
for all $n\geq 0$, by the definition above of the left hand side.  The  term $\rad^n\wM_K$ on the right hand side may be viewed as constructed in $\wA_K$-mod or in $\wfa_K$-mod, using (\ref{radicals}).  This defines a natural map $\wgr\wM\to\gr\wM$. 
When $\wM$ is $\wfa$-tight, the inclusion (\ref{tightn}) is an equality (by the $\wfa$-version of (\ref{rad})).
 This gives a physical equality between $\wgr\wM$ and $\gr\wM$ as $\gr\wfa$-modules (and certainly
 a surjection).

Conversely, assume the above natural map is surjective. Choose $n\gg 0$ so that (\ref{tightn}) is an equality;
for example, we can guarantee that both sides are 0. Using the fact that $(\wgr \wM)_{n-1}$ maps onto
$(\gr\wM)_{n-1}$, we
can conclude that equality in (\ref{tightn}) also holds for $n-1$. The converse now follows by an evident induction.
\end{proof}   
  We will see later in Corollary \ref{tightcor}  that the following lemma  holds if $\Lambda$ is replaced by any poset
  ideal. However, the lemma is used implicitly in the proof of the corollary.
  
 \begin{lem}\label{tight}Let $\wP$ be a projective indecomposable module for  $\wA=\wA_\Lambda$.
 Then $\wP$ is $\wfa$-tight.  Equivalently, $\wgr\wP=\gr\wP$.\end{lem}

 \begin{proof}
 If $\Lambda'$ is a larger poset containing $\Lambda$ as a
 ideal, we can write $\wP=\wP'_\Lambda$ where $\wP'$ is a projective indecomposable for the quasi-hereditary algebra $\wA'$ associated with the
 larger poset $\Lambda'$. By construction the algebra $\wfa$ means the same in $\wA$ and in $\wA_{\Lambda'}$. It is easy to arrange that $\wP'$ is $\wfa$-projective. (For example, take $\Lambda'$
 to be the poset of $p^r$-bounded weights in the sense of \cite[II.11]{JanB} and take $\overline{\wP'}$
 to be a $G_r$-projective indecomposable. Here $r$ is taken large enough so that $\Lambda\subseteq\Lambda'$.)

 Thus, $\wP'$ is $\wfa$-tight. Consider the commutative diagram
 
 \medskip
 \hskip1in\begin{picture}(130,80)
\put(1,55){$ \gr\wP' $}
\put(45,55){$(\gr\wP')_\Lambda$}
\put(104,55){$\gr(\wP'_{\Lambda})$}
\put(0,0){$\wgr \wP'$}
\put(99,0){$\wgr(\wP'_\Lambda)$}

\put(26,3){\vector(1,0){70}}

\put(10,15){\vector(0,1){35}}
\put(115,15){\vector(0,1){35}}

\put(23,59){\vector(1,0){17}}
\put(26,59){\vector(1,0){17}}

\put(82,59){\vector(1,0){17}}
\put(86,61){$\sim$}
\end{picture} 
  \medskip
  
  \noindent
  The right hand map on the upper horizontal row is an isomorphism using (\ref{closelyrelatedresult}), after changing the names
 of the posets.  The left hand map on the upper horizontal row is surjective by its definition. The left hand vertical map is an isomorphism by Lemma \ref{pretight}. Thus, the right hand vertical map is a surjection. Now apply Lemma \ref{pretight} again. \end{proof}
 
 \begin{cor}\label{Aistight} Let $\wA=\wA_\Lambda$.
 
 (a)  $\wA$ is a tight $\wfa$-module, i.~e.,
 \begin{equation}\label{tightequal}
 (\wrad^n\wfa)\wA=\wrad^n\wA=\rad^n\wA_K\cap\wA=(\rad^n\wfa_K)\wA_K\cap\wA,\quad \forall n\geq 0. 
 \end{equation}
  Thus, $\wgr\wA=\gr\wA$ as (graded) $\sO$-modules.  In particular,
 $\wgr\wA$ is naturally an $\sO$-algebra, isomorphic (and even equal) to $\gr\wA$. 
   
 (b) For $n\geq 0$, $\wA(\wrad^n\wfa)=(\wrad^n\wfa)\wA$.  This also (directly) gives $\wgr\wA$ an $\sO$-algebra
 structure, necessarily agreeing with that in part (a) above. In addition, for any $\wA$-module $\wM$,
 $\wgr\wM$ is a natural $\gr\wA=\wgr\wA$-module. If $\wM$ is $\wA$-lattice, then the natural map $\wgr\wM\to\gr\wM$ of $\gr\wfa$-modules  in Lemma \ref{pretight} is a map of $\gr\wA$-modules. 
 
 (c) $\wgr A$ is a $k$-algebra, naturally isomorphic to $(\wgr\wA)_k$. If $\wM$ is any $\wA$-module, then
 $(\wgr\wM)_k$ is naturally as $\wgr A$-module. In particular, if $M$ is an $A$-module, then $\wgr M$
 is naturally a $\wgr A$-module.
 \end{cor}
 
 \begin{proof} Part (a) follows because $\wA$ is a direct sum of projective indecomposable modules $\wP$
 which are all $\wfa$-tight by Lemma \ref{tight}. Hence, $\wA$ is tight.  
 
 Now we prove part (b). By (a), $\wA$ is $\wfa$-tight, so that
 $(\wrad^n\wfa)\wA=\wrad^n\wA=\rad^n\wA_K\cap\wA$ as a left $\wA$-submodule of $\wA$. Thus, 
$$\wA(\wrad^n\wfa)\subseteq\wA(\wrad^n\wfa)\wA=(\wrad^n\wfa)\wA.$$
The reverse containment holds by an evident argument working with right modules.
This verifies the first assertion in (b), and we leave the other assertions to the reader.

Finally, part (c) follows easily from part (b).
 \end{proof}
 
This leads to the following result.

\begin{cor}\label{tightcor} Let $\wM$ be an $\wA$-lattice. 

(a) $\wM$ is $\wA$-tight if and only if $\wM$
is $\wfa$-tight. 

(b) There is a natural map 
\begin{equation}\label{map} \wgr \overline{\wM}\longrightarrow\overline{\gr\wM}\end{equation}
factoring in grade $n\in\mathbb N$ as
\begin{equation}\label{factoring} (\wgr \overline{\wM})_n\cong\frac{(\wrad^n\wfa)\wM+\pi\wM  }{(\rad^{n+1}\wfa)\wM+\pi\wM  }
\longrightarrow\frac{\wrad^n\wM+\pi\wM}{\wrad^{n+1}\wM+\pi\wM}
\cong(\overline{\gr\wM})_n.\end{equation}

(c) If $\wM$ is $\wA$- (or $\wfa$-) tight, then $\wgr \wM=\gr \wM$. Also, the map (\ref{map}) is 
an isomorphism
\begin{equation}\label{identity}\wgr\overline{\wM}\cong\overline{\gr\wM}.\end{equation}

(d) Conversely, if (\ref{map}) gives an isomorphism  (\ref{identity}), then 
$\wM$ is $\wfa$-tight.
\end{cor}

\begin{proof} For part (a), Corollary \ref{Aistight}(a)  says that $(\wrad^n\wfa)\wA=\wrad^n\wA$, for all $n$, and hence that $(\wrad^n\wfa )\wM=(\wrad^n\wA)\wM$ for all $n$. Also, $\wrad^n\wM=\wM\cap\rad^n\wM_K$ agrees with its counterpart for $\wfa$, since $(\rad\wfa_K)\wA_K=\rad\wA_K$ by (\ref{radicals}). Hence,  $\wM$ is $\wfa$-tight if and only if it is $\wA$-tight. This proves (a).

Proof of (b): The left hand isomorphism in (\ref{factoring}) is obtained from the definition of $\wgr$ and the
``first isomorphism theorem." The middle map is obtained from the inclusion $(\wrad^n\wfa)\wM\subseteq
\wrad^n\wM$ and its $n+1$-analog. The right hand isomorphism is a consequence of the purity of
$\wrad^n\wM$ in $\wM$. This proves (b).

Next, consider part (c). First, the meaning of $\wgr\wM$ can be defined using $\wfa$ or $\wA$ by
Corollary \ref{Aistight}(a). This gives the first assertion of (c). Next, by above, $(\wrad\wfa^n)\wM$ is 
a pure submodule of $\wM$, so that $\overline{(\wrad\wfa^n)\wM}=(\wrad\wfa^n)M$, writing $M:=\overline{\wM}$. 
Thus, $$\overline{(\gr\wM)_n}\cong\overline{(\wgr \wM)_n}\cong\overline{\left( \frac{(\wrad^n\wfa)\wM}{(\wrad^{n+1}\wfa)\wM}\right)}
\cong\frac {\overline{(\wrad^n\wfa) \wM}}{\overline{(\wrad^{n+1}\wfa)\wM}}\cong\frac {{(\wrad^n\wfa) M}}{{(\wrad^{n+1}\wfa) M}},$$
which implies the second assertion in part (c).

To prove part (d), observe there is always a surjection $\wgr \wM\twoheadrightarrow \wgr M$ (where, as
above, $M:=\overline{\wM}$) induced by the natural homomorphism $\wM \twoheadrightarrow M$.
We can also map $\gr \wM$ to its reduction mod $\pi$, obtaining, together with (\ref{map}) and the
map of Lemma \ref{pretight}, a commutative diagram
$$\begin{CD}  \wgr\wM @>>> \gr\wM\\
@VV{\text{onto}}V @VV{\text{onto}}V\\\wgr M @>>>
\overline{\gr\wM}.\end{CD}$$
Now part (d) follows from Nakayama's lemma for $\sO$-modules
and Lemma \ref{pretight}.
\end{proof}
 
 \begin{cor}\label{cor3.9} For each $\lambda\in\Lambda$, the module $\wDelta(\lambda)$ is $\wfa$-tight. More
 generally, if $\Gamma\subseteq\Lambda$ is a poset ideal, and if $\wP$ is a projective indecomposable
 module in $\wA_\Gamma$, then $\wP$ is $\wfa$-tight, as is $\wA_\Gamma$. \end{cor}

 \begin{proof}By Corollary \ref{tightcor}, it suffices to show that $\wA_\Gamma$ is $\wA$-tight.  In fact, given $\lambda\in \Lambda$, 
 let $\Gamma$ be the ideal in $\Lambda$ generated by $\lambda$. Then $\lambda$ is maximal in
 $\Gamma$, so that $\wDelta(\lambda)$ is a projective indecomposable module in $\wA_\Gamma$.
 
Now let $\Gamma$ be any poset ideal in $\Lambda$, and let $\wP$ be a projective indecomposable module in $\wA_\Gamma$--mod.
 Form 
 a commutative diagram
$$\begin{CD}
 \wgr \wA @= \gr\wA\\ @V\text{onto}VV @VV\text{onto}V \\
 \wgr\wA_\Gamma @>>> \gr\wA_\Gamma
 \end{CD}$$
 of $\sO$-modules and natural maps. By (\ref{closelyrelatedresult}), 
 $\gr 
 \wA_\Gamma\cong (\gr\wA)_\Gamma$, which implies that the right hand vertical map is surjective. On the other hand, the left hand vertical map is clearly surjective by the definition of $\wgr$. Therefore, the bottom horizontal map
 is surjective, and Lemma \ref{pretight} implies that $\wA_\Gamma$ is tight.
  \end{proof}

  \section{A homological criterion for $p$-filtrations}
\setcounter{equation}{0}

   {\it In this section, Assumptions \ref{assumptions}
  are in force.} In particular, the prime $p$ is odd and $\geq 2h-2$. Let $\Gamma$ be a non-empty ideal in the finite poset $\Lambda$ which is required to
  be itself an ideal in the poset  of regular dominant weights. 
 Our focus will be on $\Gamma$ and on modules for the algebra $A_\Gamma$. But we need
 $\Lambda$ large enough so that various related rational
  $G$-modules make sense as modules for  $A_\Lambda$. 
 {\it For that reason, we also
 assume that $\Lambda$ satisfies the condition (\ref{w0}) for 
all  $\gamma\in\Gamma$.} This condition means that if $\gamma\in\Gamma$, then
$\wP^\sharp(\gamma)$ and $\wQ^\sharp(\gamma)$ belong to $\wA_\Lambda$--mod.
See (\ref{wP}) and (\ref{wQ}).
 Recall from the previous section that $\wfa$ is the image of the (integral) small quantum group $\widetilde u_\zeta$ in $\wA_\Lambda$.   Also,
 $\fa:=\wfa_k$ is isomorphic (through the natural surjection) to the sum of the regular blocks in the restricted enveloping algebra $u$ of $G$.  
 
 Additional assumptions, involving the LCF, will be introduced in Lemma \ref{preliminaryLCF}(c) and after its
 proof. The LCF will not be used until that point.
     
Let $M$ be an $\fa$-module. For a non-negative integer $r$, let, in the notation of (\ref{decreasing}),
\begin{equation}\label{titlder}M_{\widetilde r}:=(\wgr M)_r=\wF^rM/\wF^{r+1}M=(\wrad^r\wfa) M/(\wrad^{r+1}\wfa) M.\end{equation}
(This is a slight abuse of notation. We will not use the symbol ``$\widetilde r$", for an integer $r$, except as a subscript as
above.)
 If $M$ is an $A$-module, then $M_{\widetilde r}$ is also an $A$-module (with $\wF^{>0} A:=\Sigma_{n>0}F^nA$ acting trivially).

 Put $\wsoc^{-n}M:=\{x\in M\,|\, \wF^nx=0\}$, and set
 \begin{equation}\label{grdiamond} \wgr^\diamond M:=\bigoplus_{n\geq 0}\wsoc^{-n-1}M/\wsoc^{-n}M
 \end{equation}
 with the index $n$ giving the degree $-n$ term. (Thus, $\wgr^\diamond M$ is negatively graded.)
 If $M$ is an $A$-module, then $\wgr^\diamond M$ is a graded $\wgr A$-module.

\begin{lem}\label{preliminaryLCF}Let $\mu\in\Gamma$.

(a) $\nabla_p(\mu)$ is naturally isomorphic to the $\fa$-socle of $Q^\sharp(\mu)$ and also naturally
isomorphic to the $\wgr \fa$-socle of $\gr^\diamond Q^\sharp(\mu)$.

(b) There are inclusions
$$\nabla_p(\mu)\subseteq\rnabla(\mu)\subseteq \wsoc^{-1}Q^\sharp(\mu)$$
of $A$-modules.

(c) Assume the LCF holds for $G$. Then the inclusions in part (b) are equalities, for all $\mu\in\Gamma$.
 \end{lem}

\begin{proof} The first claimed isomorphism in part (a) is clear from the definitions.  For the second isomorphism,
it suffices to prove the dual statement that the $\fa$-head of $P^\sharp(\mu)$ is isomorphic to the 
$\wgr\fa$-head of $\wgr P^\sharp(\mu)$. (Both heads are $\fa/(\wrad\wfa)\fa  =(\wgr\fa)_0$-modules.) The
$\fa$-head of $P^\sharp(\mu)$ is $P^\sharp(\mu)/\rad\fa P^\sharp(\lambda)$, a homomorphic image
of $\wgr P^\sharp(\mu)_0$. To check that this natural homomorphism induces an isomorphism on
heads, it is enough to check the corresponding statement with $P^\sharp(\mu)$ replaced by $\fa$, since
$P^\sharp(\mu)$ is $\fa$-projective. Here,
however, it is obvious. Thus, (a) is proved.

To see  part (b), observe that the first inclusion follows from the definitions. Next, we assert that $\rnabla(\mu)\subseteq Q^\sharp(\mu)$. We can assume that $\mu\in X_1(T)$, using Proposition \ref{Lin}(a). 
By Proposition 2.3(b), $Q^\sharp(\mu)$ has a $\rnabla$-filtration and $G$-socle $L(\mu)$. Thus, $\nabla(\mu)
$ is a submodule of $Q^\sharp(\mu)$. On the other hand, $\rnabla(\mu)\subseteq \nabla(\mu)$, proving
our assertion. 
 Finally,  $\wrad\fa$
acts trivially on $\rnabla(\mu)$, since this module arises by base change from a lattice in
an irreducible $\wA_K$-module. Thus, the second containment holds. 

Finally,  observe that  if LCF holds for $G$, then
 the heads of $\wfa_K$ and $\fa$ have the same dimensions.
  Thus, the nilpotent ideal $(\wrad\wfa)_k$ has codimension equal to that of $\rad\fa$, so must be equal to $\rad\fa$. Therefore, $\wsoc^{-1} Q^\sharp(\mu)$ is the $\fa$-socle of $Q^\sharp(\mu)$, namely, $\nabla_p(\mu)$. This clearly implies (c).
\end{proof} 

{\it For the rest of this section, it is assumed (in additions to Assumptions \ref{assumptions} and (\ref{w0})  that the LCF holds for $G$.  Equivalently, by Corollary \ref{LCF}, $\rDelta(\lambda)=\Delta^p(\lambda)$, for all
$\lambda\in X(T)_+$.}

 Recall the convention that, if $M$ is a $\mathbb Z$-graded module, then $M(a)$ is the $\mathbb Z$-graded
 module defined by putting
$M(a)_b:=M_{b-a}$, for integers $a,b$. 

 \begin{lem}\label{numlemma} Let $M\in A_\Gamma$--mod, $\mu\in\Gamma$, and $r\in\mathbb N$. Assume,  for each non-negative integer $s<r$, that $M_{\widetilde s}$ has a $\rDelta$-filtration.
Then there is an isomorphism
$$f:\grhom_{\wgr A}(\wgr M(-r),\wgr^\diamond Q^\sharp(\mu))\overset\sim\longrightarrow \Hom_A(M_{\widetilde r},\rnabla(\mu))=\Hom_{\wgr A}(M_{\widetilde r},\rnabla(\mu))$$
given by restriction.\end{lem}

\begin{proof}  By Lemma \ref{preliminaryLCF}(c) (which uses the LCF) and the definition (\ref{grdiamond}), $(\wgr^\diamond Q^\sharp(\mu))_0$ identifies with $\rnabla(\mu)$ as $(\wgr A)_0\cong\wgr A/(\wgr A)_{>0}$-modules.  Also, the action of
$A$ on $M_{\widetilde r}
:=(\wrad^r\wfa) M/(\wrad^{r+1}\wfa)M$ factors through the quotient $A\to (\wgr A)_0$. 
It follows that there is a map  the map 
$$f:\grhom_{\wgr A}(\wgr M(-r),\wgr^\diamond Q^\sharp(\mu))\overset\sim\longrightarrow \Hom_A(M_{\widetilde r},\rnabla(\mu))$$
which is induced by restriction to the homogeneous piece $M_{\widetilde r}$ of $\wgr M(-r)$.

\medskip\noindent
\underline{The map $f$ is injective.}
To see this, observe that, since any element $x$ of the left hand
side, if nonzero, must have a nonzero image in $\wgr^\diamond Q^\sharp(\mu)$. Any such image is a $\wgr\fa$-submodule of $\wgr^\diamond Q^\sharp(\mu)$, and must intersect the $\wgr\fa$-socle
of $\wgr^\diamond Q^\sharp(\mu)$, which is contained in $\wsoc^{-1} Q^\sharp(\mu)$. (Actually, we have
equality, but an inclusion would be sufficient, as would follow just from knowing that $\nabla_p(\mu)=\soc^{-1}Q^\sharp(\mu)^\Gamma$,
the largest submodule of $\wsoc^{-1}Q^\sharp(\mu)$ with composition factors having highest weights
in $\Gamma$. This equality is also sufficient to define $f$. We will return to this issue in \S7.)
 
  But the image of $x$ is a direct sum of terms $x(M_{\tilde s})\subseteq (\wgr^\diamond Q^\sharp(\mu))_{s-r}$, with $s\geq 0$.
Thus, $x(M_{\widetilde r})\not=0$. Therefore, $f$ is injective.

\medskip
\noindent
\underline{The map $f$ is surjective.}
The hypothesis on $\rDelta$-filtrations implies
that 
$$\Ext^1_A(M_{\tilde s},Q^\sharp(\mu))=0, \quad {\text{\rm for}}\,\,s<r,$$
 by Proposition \ref{Lin}(b). Thus, by the long exact
sequence for $\Ext^\bullet_A(-,Q^\sharp(\mu))$,  there is a surjection
\begin{equation}\label{surjection}\Hom_A(M/(\wrad^{r+1}\wfa)M,Q^\sharp(\mu))\twoheadrightarrow
\Hom_A(M_{\widetilde r},Q^\sharp(\mu))\cong\Hom_A(M_{\widetilde r},\rnabla(\mu))\end{equation}
which factors through the inverse of the natural isomorphism
$$\Hom_A(M/(\wrad^{r+1}\wfa)M,\wsoc^{-r-1}Q^\sharp(\mu))\overset\sim\longrightarrow\Hom_A(M/(\wrad^{r+1}\wfa) M,Q^\sharp(\mu)).$$
Notice that 
any $\wfa$-submodule $N$ of $Q^\sharp(\mu)$ having $(\wrad^{r+1}\wfa)N=0$ must be contained in $\wsoc^{-r-1}Q^\sharp(\mu)$.

Now, given $g\in\Hom_A(M_{\widetilde r}, \rnabla(\mu))$, also regarded as a map to $Q^\sharp(\mu)$, 
we can therefore choose
$$h\in\Hom_A(M/(\wrad^{r+1}\wfa)M,Q^\sharp(\mu)),\quad h|_{M_{\widetilde r}}=g.$$
 Of
course,  $h'|_{M_{\widetilde r}}=g$, where $h'\in\Hom_A(M/(\wrad^{r+1}\wfa)M,\wsoc^{-r-1} Q^\sharp(\mu))$ is induced by $h$. We also have $(\wgr h')|_{M_{\widetilde r}}=g$, where
$\wgr h':\wgr(M/(\wrad^{r+1}\wfa)M)\longrightarrow\wgr(\wsoc^{-r-1} Q^\sharp(\mu))$ given by applying the $\wgr$-functor. Here, by a mild abuse of
notation,  $g$ can be identified with the map it induces $\wgr(M/(\wrad^{r+1}\wfa) M)_r\cong M_{\widetilde r}\longrightarrow \wgr(\wsoc^{-r-1} Q^\sharp(\mu))_r\subseteq
\wsoc^{-1} Q^\sharp(\mu)\cong\rnabla(\mu)$. 

Consider the diagram 
\[
\begin{CD} \wgr M(-r) @>\tau>> \wgr(M/(\wrad^{r+1}\wfa)M)(-r) \\
@.  @VV(\wgr h')(-r)V\\
\wgr^\diamond(\wsoc^{-r-1}Q^\sharp(\mu)) @<\theta<< \wgr(\wsoc^{-r-1} Q^\sharp(\mu))(-r)\\
@V\sigma VV @. \\
\wgr^\diamond Q^\sharp(\mu) 
\end{CD}
\]
in which  $\tau$ is obtained by applying
$\wgr$ to the quotient map $M\to M/(\wrad^{r+1}\wfa)M$, $\theta$ is the obvious natural map of $\wgr A$-modules, and $\sigma$ is the natural inclusion.
The composition of the maps defines a graded map 
 $$h^{\prime\prime}:=\sigma\circ\theta\circ (\wgr h')(-r)\circ\tau:\wgr M(-r)\to\wgr^\diamond Q^\sharp(\mu)$$
 which, in grade 0---that is, on $M_{\widetilde r}\cong(\wgr M(-r))_0$---identifies with $g$. Therefore, $f$ is surjective.
 
 It follows that $f$ is an isomorphism.
 \end{proof}
 
   We are now ready for the main theorems of this section. As its proof shows, the first theorem below is a formality for
  general QHAs, though it is formulated for our context here.

\begin{thm}\label{secondmainnum}Suppose that $N\in\Amod$ is annihilated by $\wrad\wfa$. Then $N$ has a $\rDelta$-filtration if and only
if
\begin{equation}\label{numformula}\dim\,N=\sum_{\mu\in\Lambda}\dim\Hom_A(N,\rnabla(\mu))\dim\rDelta(\mu).\end{equation}

 \end{thm}
 
 \begin{proof} By Corollary \ref{ZeroDegreeQHA}, 
$\wA/\wrad\wA=(\gr \wA)_0$ is an integral QHA with standard (resp., costandard) objects $\wrDelta(\lambda)$ (resp., $\wrnabla(\lambda)$).  Because $\wA$ is $\wfa$-tight, 
$$\overline{\wA/\wrad\wA}
\cong A/\wrad\wfa A$$
 by (\ref{identity}) in Corollary \ref{tightcor}. Now $N$ has a semistandard filtration in the sense of \cite{S2}, with the multiplicity with which a
nontrivial homomorphic image of a given $\rDelta(\mu)$ appears equal to $\dim\Hom_A(N,\rnabla(\mu))$. Thus,
$$\dim N\leq\sum_{\mu\in\Lambda}\dim\Hom_A(N,\rnabla(\mu))\dim\rDelta(\mu)$$ 
Equality holds if and only if each homomorphic image of a $\rDelta(\mu)$ appearing is actually isomorphic to $\rDelta(\mu)$, in which case
the filtration is a $\Delta$-filtration for the QHA $A/\wrad \wfa A$, i.~e., a $\rDelta$-filtration.
 \end{proof}

An $A$-module $N$ is said to have a tight lifting if $N\cong\overline{\wN}$ for some tight $\wA$-lattice $\wN$.  This implies that $\wgr N\cong {\overline{\gr\wN}}$ by Corollary \ref{tightcor}(c). Also, we 
recall from Corollary \ref{tightcor}(a) that any $\wA$-lattice is $\wA$-tight if and only if it is $\wfa$-tight. Notice if $N$ is an $A_\Gamma$-module, with a tight lifting or not, and $\wN$ is an $\wA_\Lambda$-lattice with $\overline{\wN}\cong N$, then $\wN$ is necessarily
an $\wA_\Gamma$-lattice.  ($\wN$ has a filtration with sections which are lattices $\wL_\lambda$ in irreducible
$\wA_K$-modules $L_K(\lambda)$, $\lambda\in\Lambda$. Then $L(\lambda)$ is always a
composition factor of $\overline\wL_\lambda$, forcing $\lambda\in\Gamma$.) Thus, any lift on $N$, tight
or not, is necessarily an $\wA_\Gamma$-lattice. 
  Also, $\gr\wN$ (defined using $\wA_\Lambda$) is a $\gr\wA_\Gamma=
(\gr\wA)_\Gamma$-module.
\begin{thm}\label{thm4.4}
Suppose an $A_\Gamma$-module $M$ has an $\wA$-tight lifting to a lattice $\wM$ for $\wA$. Then $M_{\widetilde s}$ has a $\rDelta$-filtration, for each $s\geq 0$, if and only if
\begin{equation}\label{Donkin}\Ext^1_{\gr\wA}(\gr\wM,\gr^\diamond \wQ^\sharp(\mu))=0,\quad\forall
\mu\in\Gamma.\end{equation}\end{thm}

\begin{proof}First, $(\gr^\diamond\wQ^\sharp(\mu))_K$ is $\gr \wA_K=(\gr\wA)_K$-injective, so that $\Ext^1_{\gr \wA}(\gr \wM,\gr^\diamond \wQ^\sharp(\mu))$ is an $\mathcal O$-torsion module. Hence, it is zero if and only if the reduction mod $\pi$
 map
 \begin{equation}\label{surj2}\Hom_{\gr\wA}(\gr\wM,\gr^\diamond \wQ^\sharp(\mu))\to\Hom_{\gr\wA}(\overline{\gr \wM},\overline{\gr^\diamond \wQ^\sharp(\mu)})\end{equation}
 is surjective. By the tightness hypothesis and Corollary \ref{tightcor}(c), $\overline{\gr\wM}\cong\wgr M$.  Also,
 $$\overline{\gr^\diamond \wQ^\sharp(\mu)}\cong\wgr^\diamond  Q^\sharp(\mu).$$
To see this isomorphism, first notice that, since 
 $\wP^\sharp(\mu)$ is a tight $\wfa$-lattice, we have $\overline{\gr\wP^\sharp(\mu)}=\wgr P^\sharp(\mu)$.
 Apply the duality functor $\frak d$ to both sides. On the left side, we get, using the discussion above Corollary
 \ref{cor3.3},
 $$\frak d(\overline{\wP^\sharp(\mu)})\cong\overline{\wdelta(\gr\wP^\sharp(\mu))}\cong\gr^\diamond\wdelta(
 \wP^\sharp(\mu))\cong\overline{\gr^\diamond\wQ^\sharp(\mu)}.$$
 Next, use the 
 general fact that, for any $A$-module $N$, there is a natural isomorphism ${\mathfrak d}(\wgr N)\cong\wgr^\diamond{\frak d}(N)$; we leave the easy proof to the
 reader. Thus, if $\frak d$ is applied to the right hand side $\wgr P^\sharp(\mu)$, we get
 $\wgr^\diamond{\frak d}( P^\sharp(\mu))\cong\wgr^\diamond Q^\sharp(\mu)$, as desired.

Now assume the $\Ext^1_{\gr \wA}$-vanishing hypothesis of the theorem holds for all $\mu\in\Gamma$. 
 Then, taking into account the grading, (\ref{surj2}) gives a surjection
 \begin{equation}\label{surj3}\grhom_{\gr\wA}(\gr\wM(-s),\gr^\diamond\wQ^\sharp(\mu))\twoheadrightarrow\grhom_{\gr \wA}(\wgr M(-s),\wgr^\diamond Q^\sharp(\mu)),\end{equation}
 for each integer $s$ and each $\mu\in\Gamma$. Observe that 
 $$\dim\grhom_{\gr A_K}(\gr\wM_K(-s),\gr^\diamond\wQ^\sharp(\mu))_K)= [(\gr\wM_K)_s:L_K(\mu)]$$
 since $(gr^\diamond\wQ^\sharp(\mu))_K$ is the injective envelope of $L_K(\mu)$ in $\gr\wA_K$--mod. 
 Thus, the  $\sO$-lattice $\grhom_{\gr\wA}(\gr\wM(-s),\gr^\diamond\wQ^\sharp(\mu))$ has
 rank  $[(\gr\wM_K)_s:L_K(\mu)]$.  By general principles (see, e.~g., \cite[Lem. 1.5.2(b)]{CPSMem}), for $\gr\wA$ lattices $\wX$ and $\wY$, we have  $\overline{\Hom_{\gr \wA}(\wX,\wY)}
 \subseteq\Hom_{\gr \wA}(\overline{\wX},\overline{\wY})$, so
 $$\dim\grhom_{\gr A}(\wgr M(-s),\wgr^\diamond Q^\sharp(\mu))=[(\gr\wM_K)_s:L_K(\mu)],$$
 since the map (\ref{surj3}) is a surjection.
 
 So we get, if $r$ is an integer satisfying the hypothesis of Lemma \ref{numlemma},
  $$\begin{aligned} \dim(\gr\wM_K)_r &=\sum_{\mu\in\Gamma}[(\gr\wM_K)_r:L_K(\mu)]\dim L_K(\mu)\\ &=\sum_{\mu\in\Gamma}\dim\Hom_A(M_{\widetilde r},\rnabla(\mu))\dim\rnabla(\mu).\end{aligned}$$
Here we are using the fact that $\wM$ is an $\wA_\Gamma$-lattice, by the discussion before the theorem.  
 Since $\dim(\gr\wM_K)_r=\dim M_{\widetilde r}$, an induction on $r$ proves that each $M_{\widetilde r}$ has a $\rDelta$-filtration. 
 
 Conversely, assume that each $M_{\widetilde s}$ has an $\rDelta$-filtration. Tracing back through the above discussion recovers the $\Ext^1$-vanishing in the statement of the theorem. We leave further
 details to the reader.\end{proof}
 
 Finally, he following result  is a consequence of Theorem \ref{thm4.4}, since the $\Ext^1$-vanishing in Theorem \ref{firstmainnum} implies the
 vanishing condition in Theorem \ref{thm4.4}; see \cite[Lem. 1.5.2(c)]{CPSMem}.
 
 \begin{thm}\label{firstmainnum}Suppose the $A_\Gamma$-module $M$ has a tight lifting. Then $M_{\tilde s}$ has a $\rDelta$-filtration, for each $s\geq 0$, provided that
$$\Ext^1_{\wgr A}(\wgr M,\wgr^\diamond Q^\sharp(\mu))=0,\quad\forall\mu\in\Gamma.$$
\end{thm}

 \section{$p$-filtrations of Weyl modules}
\setcounter{equation}{0}

Let $G$ be a simple, simply connected algebraic group over an algebraically closed field $k$ of positive characteristic $p$. Making use of the results so far obtained, we prove the following theorem. The cases $s=0,1$ are consequences of earlier work in \cite{PS3} and \cite{CPS7}, which inspired the present paper.

\begin{thm}\label{MainTheorem}Assume that $p\geq 2h-2$ is an odd prime and that the LCF holds for $G$.
Given any $\gamma\in \Xreg $,  each section 
$\Delta(\gamma)_{\widetilde s}$, $s\in\mathbb N$, viewed as a rational $G$-module, has a $\rDelta$-filtration.  In particular,
$\Delta(\gamma)$ has a $\rDelta=\Delta^p$-filtration. \end{thm}

\begin{proof} It suffices to verify the hypothesis of Theorem \ref{thm4.4} for $M=\Delta(\gamma)$. First,
Corollary \ref{cor3.9} implies that $\wDelta(\gamma)$ is an $\wfa$-tight lifting of $\Delta(\gamma)$.
Thus, $\wDelta(\gamma)$ has a $\wA$-tight lifting by Corollary \ref{tightcor}. Finally, to check 
the $\Ext^1$-condition (\ref{Donkin}), observe that Corollary \ref{cor3.3} implies the $\gr\wA$-module
$\gr^\diamond\wQ^\sharp(\mu)$, $\mu\in\Gamma$, has a filtration by costandard modules (namely, the
various $\gr\wnabla(\tau)$), for the QHA $\gr\wA$. Thus, the $\Ext^1$-group vanishes.\end{proof}

\begin{rem}\label{rem5.2} Section 7 shows, in view of the above proof, that it is enough in Theorem \ref{MainTheorem}
to assume that the LCF holds for any weight $\gamma\in \Xreg $ satisfying $\gamma<\lambda$.
\end{rem}

By \cite[Lemma 3.1(b)]{CPS7}, any Jantzen translation functor carries $\Delta^p(\lambda)$, $\lambda\in\Xreg$, to either
the 0 module or to another module of the form $\Delta^p(\mu)$, $\mu\in X(T)_+$. Since these functors are exact, there is the following consequence.

\begin{cor} Assume that $p\geq 2h-2$ is an odd prime and that the LCF formula holds for $G$. Then for any $\gamma\in X(T)_+$, $\Delta(\gamma)$ has a $\rDelta=\Delta^p$-filtration.
\end{cor} 

 Using the following lemma, Theorem \ref{MainTheorem} can be recast, using $G_1$-radical series. We
 continue to assume in the rest of this section that $p\geq 2h-2$ is an odd prime and that the LCF holds for $G$. For
 $\mu_0\in X_1(T)$, let $Q_1(\mu_0)=\widehat Q_1(\mu_0)|_{G_1}$. Also, $u=u({\mathfrak g})$ is the
 restricted enveloping algebra of the Lie algebra $\mathfrak g$ of $G$. 

\begin{lem}\label{radicalseriesofintegral quantum}  Let $\mu_0\in X_1(T)\cap \Xreg $. For $n\geq 0$,
\begin{equation}\label{rigid}\overline{\wQ_\zeta(\mu_0)\cap\rad^n_{u_\zeta}Q_\zeta(\mu_0)} \cong
\rad^n_{u({\mathfrak g})}Q_1(\mu_0).\end{equation}
\end{lem}
\begin{proof} By choosing a poset $\Lambda$ of regular dominant weights  large enough, we can that 
$\wQ_\zeta(\mu_0)$ is an $\wA=\wA_\Lambda$-module, which we denote by $\wQ^\sharp(\mu_0)$
in keeping with previous notation. First, observe that $\overline{\wrad\fa}=\rad\fa$, as was shown in the proof of Lemma \ref{preliminaryLCF}(c). Of course, $(\wrad\wfa)(\wrad^n\wfa)\subseteq\wrad^{n+1}\wfa$ for any $n\in\mathbb N$. Since $\wQ^\sharp(\mu_0)$ is tight, the left hand side of (\ref{rigid}) defines a filtration of
$Q^\sharp(\mu_0)=Q_1(\mu_0)$ whose sections are $\wrad^n\wfa(Q^\sharp(\mu_0)$. It thus defines a semisimple
series of $Q^\sharp(\mu_0)$ whose length is at most the length of a radical series of $\wQ_K(\mu_0)$.
The right hand side of (\ref{rigid}) defines the radical series of $Q_1(\mu_0)$.  Also, the left hand
side terms give a semisimple series, having length equal to   that of the radical series of $Q_1(\mu_0)$. 
  On the other hand,  $Q_\zeta(\mu_0)$ and $Q_1(\mu_0)$ are both rigid modules having Loewy length equal to $|\Phi| +1$ by \cite[II.D.14\&D.8]{JanB} (and its evident quantum analog).
An elementary argument establishes the filtrations of $Q_1(\mu_0)$ must be the same, so that
(\ref{rigid}) holds. \end{proof}

\begin{rem} An alternate proof can be given using the Koszulity
of the regular weight projections of the small quantum group and restricted enveloping algebra, proved in \cite{AJS} under the assumptions that $p>h$ and that
the LCF holds for $G$. These assumptions
hold here. It follows that the multiplicity of an irreducible module in the
$n$th radical layer of  $\wQ_\zeta(\mu_0)$ or of $Q_1(\mu_0)$ can be computed using
products of the same (analogs of) inverse Kazhdan-Lusztig polynomials. The
 result now follows from the remarks at the beginning of the proof above, which show the right hand
side of (\ref{rigid}) is contained in the left hand side. So a codimension count shows equality must hold.
 \end{rem}

For an $\fa$-module $M$, we put (for emphasis) $\gr_\fa M=\bigoplus_{n\geq 0}\rad^nM/\rad^{n+1}M$, where
here $\rad^nM=(\rad \fa)^nM$. 
As a consequence, we immediately obtain the following important result.
\begin{cor}\label{rigid2} For $n\in\mathbb N$, 
$$\overline{\wrad^n\wfa}=(\rad\fa)^n.$$
Thus, for any $\fa$-module $M$, we have
$\wgr M= \gr_{\fa}(M).$
\end{cor}

\begin{thm}\label{firsttheorem}
For any $\lambda\in \Xreg $, each section of the $G_1$-radical series
of $\Delta(\lambda)$ has a $\rDelta$-filtration. In particular, each section of this radical series
has a $\Delta^p$-filtration.\end{thm}

\begin{proof}By Corollary \ref{rigid2},
 $\wgr \Delta(\lambda)=\gr_\fa\Delta(\lambda)$ for any regular dominant weight $\lambda$. Thus, for any positive integer $s$, $\Delta(\lambda)_s
=\Delta(\lambda)_{\widetilde s}$. The theorem thus follows from Theorem \ref{MainTheorem}. \end{proof}

\section{Appendix I}
\setcounter{equation}{0}

 Let $(K,\sO,k)$ be a $p$-local system, and let $\wA$ be an integral quasi-hereditary
algebra over $\sO$. Assume that $\Lambda$ is the poset of the QHA $\wA$, and for $\lambda\in\Lambda$,
the standard module $\wDelta(\lambda)$ and costandard module $\widetilde\nabla(\lambda)$ are given.
Both $\wDelta(\lambda)$ and $\widetilde\nabla(\lambda)$ are $\wA$-lattices.

\begin{prop}\label{filtrationprop}Let $\wM$ be an $\wA$-lattice with the property that 
\begin{equation}\label{tooty}\Ext^1_\wA(\wM,\widetilde\nabla(\lambda))=0\end{equation}
for all $\lambda\in\Lambda$. Then $\wM$ has a $\wDelta$-filtration.\end{prop}

\begin{proof} For any given $\lambda$, the vanishing condition (\ref{tooty}) is equivalent to both the following
two conditions holding:
\begin{equation}\label{equivalence}\begin{cases} (1)\quad\Ext^1_{\wA_K}(\wM_K,\nabla_K(\lambda))=0; \\
(2)\quad \Hom_\wA(\wM,\widetilde\nabla(\lambda))\to \Hom_{\wA}(\wM,\nabla(\lambda))\,\,{\text{\rm is
surjective.}}\end{cases}\end{equation}
(The map in (2) is induced by the natural map $\wnabla(\lambda)\to\nabla(\lambda)$.) This fact is an easy application of the long exact sequence of cohomology associated to the short exact sequence
$0\to\wnabla(\lambda)\overset\pi\longrightarrow\wnabla(\lambda)\to\nabla(\lambda)\to 0$. We leave the details  to the reader. Furthermore, (\ref{grExt}) implies that, when $\wM$ is an $\wA_\Gamma$-lattice for some poset $\Gamma$ in $\Lambda$, and $\lambda\in\Gamma$, then each condition (1) and (2) above is equivalent to
its counterpart over $\wA_\Gamma$.  

For any nonempty poset ideal $\Gamma$ in $\Lambda$, if 
$\wX$ is an $\wA$-lattice such that all composition factors $L_K(\nu)$ of $\wX_K$ satisfy
$\nu\in\Gamma$, then $\wX$ is an $\wA_\Gamma$-module. In particular, any irreducible
section (``composition factor") of $\wX$ has the form $L(\gamma)$ with $\gamma\in\Gamma$. 
 
  We will prove the result by induction on $|\Lambda|$, the case $\Lambda=\emptyset$ and $\wA=0$
  being trivial.
 Let $\lambda\in\Lambda$ be maximal with
$L_K(\lambda)$ a composition factor of $\wM_K$.  Put $\Gamma=\{\gamma\in\Lambda\,|\,
\gamma\not>\lambda\}$. Thus, $\Gamma$ 
is a poset in $\Lambda$ containing $\lambda$ as a maximal element, and $\wM$ is an $\wA_\Gamma$-lattice. We can therefore assume, without loss,  that $\Gamma=\Lambda$.

  Let $N_K$ be the largest quotient of $\wM_K$ for
which $[N_K:L_K(\lambda)]=0$. Set $\wD$ be the kernel of the map $\wM\subseteq\wM_K\twoheadrightarrow
N_K$. Then $\wD$ is a pure submodule of $\wM$. Let $\wN=\wM/\wD$, an $\wA$-lattice for $N_K$. Then
$L(\lambda)$ is not a composition factor of $\wN$. For any $\nu\not=\lambda$, the map
$\Hom_\wA(\wM,\widetilde\nabla(\nu))\to\Hom_\wA(\wM,\nabla(\nu))$ is surjective, since $\Ext^1_\wA(\wM,\widetilde\nabla(\nu))=0$.  Notice $\Hom_{\wA_K}(\wD_K,\nabla_K(\nu))=0$ if $\nu\not=\lambda$. (The head of $\wD_K$ has 
only composition factors $L_K(\lambda)$, by maximality of $N_K$. If $\Hom_{\wA_K}(\wD_K,\nabla_K(\nu))\not=0$, a composition factor $L_K(\mu)$ of $\nabla_K(\nu)$ would occur in the head of $\wD_K$, and also $L_K(\nu)$
would be a composition factor of $\wD_K\subseteq\wM_K$. This gives $\lambda=\mu\leq \nu\not=\lambda$, which is
a contraction to the maximality of $\lambda$.) Thus, $\Hom_\wA(\wM,\widetilde\nabla(\nu))\cong
\Hom_{\wA}(\wN,\widetilde\nabla(\nu))$. Also, by the surjectivity above, every map $f\in\Hom_\wA(\wM,\nabla(\nu))$ comes from
a map $\widetilde f\in\Hom_\wA(\wM,\widetilde\nabla(\nu))\subseteq\Hom_{\wA_K}(\wM_K,\nabla_K(\nu))$
by reduction mod $\pi$. Since $\widetilde f(\wD)=0$, $f(\wD)=0$ as well. Hence, $\Hom_\wA(\wM,\nabla(\nu))
\cong
\Hom_\wA(\wN,\nabla(\nu))$. Now we have that the natural map $\Hom_{\wA}(\wN,\widetilde\nabla(\nu)
\to\Hom_{\wA}(\wN,\nabla(\nu))$ is surjective, for all $\nu\not=\lambda$.  Put $\Lambda'=\Lambda
\backslash\{\lambda\}$, a proper poset ideal of $\Lambda$ (possibly empty). For $\nu\in\Lambda'$, consider the 
exact sequence
$$\Hom_{\wA_K}(\wD_K,\nabla_K(\nu))\to\Ext^1_{\wA_K}(\wN_K,\nabla_K(\nu))\to\Ext^1_{\wA_K}(\wM_K,\nabla_K(\nu))$$
arising from the long exact sequence for $\Ext$. The right hand end is 0 by hypothesis. The left hand
term is also 0 by construction, as previously argued.  Thus,
the middle term vanishes.  Since both of
its modules belong to $\wA_{\Lambda', K}$,  $\Ext^1_{\wA_{\Lambda'}}(\wN_K,\nabla_K(\nu))=0$
for all $\nu\in\Lambda'$. By induction, $\wN$ has a $\wDelta$-filtration.

It remains only to show that $\wD$ has a $\wDelta$-filtration. We have $$\Ext^1_{\wA}(\wN,\nabla(\nu))=
\Ext^1_A(\wN_k,\nabla(\nu))=0,\quad \forall\nu,$$
 since $\wN$ has a $\wDelta$-filtration.  (For later use, observe that these vanishings and their analogues over $\wA_K$ holds for all groups $\Ext^n$, $n\geq 1$.) Thus, every element
of $\Hom_\wA(\wD,\nabla(\nu))$ comes by restriction from $\Hom_\wA(\wM,\nabla(\nu))$, by the long
exact sequence of cohomology. But every element of $\Hom_\wA(\wM,\nabla(\nu))$, $\nu\not=\lambda$, vanishes on $\wD$ as
shown (for $f$) in the previous paragraph. Therefore, $\Hom_\wA(\wD,\nabla(\nu))=0$ for $\nu\not=\lambda$.
 Thus, the head of $\wD$ has the form $L(\lambda)^{\oplus n}$. Next, observe that,
by a long exact sequence of cohomology argument, $\Ext^1_{\wA_K}(\wD_K,\nabla_K(\nu))=0$, for all $\nu\in\Lambda$.  
Also, $\Hom_{\wA_K}(\wD_K, \nabla_K(\nu)/L_K(\nu))=0$ for all $\nu\in\Lambda$. This is because
$\lambda$ is maximal and the head of $\wD_K$ is isomorphic to $L_K(\lambda)^{\oplus m}$. 
(Here $m$ is a positive integer as is $n$ above.)
Consequently, 
$\Ext^1_{A_K}(\wD_K,L_K(\nu))=0$, for all $\nu$, by another long exact sequence argument. Therefore, $\wD_K$ is a projective $\wA_K$-module.
Hence, $\wD_K(\lambda)\cong\Delta_K(\lambda)^{\oplus m}$. Since the head of $\wD\cong L(\lambda)^{\oplus n}$, $\wD$ is a homomorphic image of $\wDelta(\lambda)^{\oplus n}$. Therefore, $n\geq m$.

If we reduce the lattice $\wD$ mod $\pi$ and count its composition factors, with multiplicities, we get the
same answer as when we reduce mod $\pi$ the lattice $\wDelta(\lambda)^{\oplus m}$ in $\wD_K$. Hence, $n\leq m$. 
Thus, equality holds, and the surjection $\wDelta(\lambda)^{\oplus n}\twoheadrightarrow\wD$ is an
isomorphism, by dimension considerations.
\end{proof}

\begin{rem} Of course, a dual argument shows that $\wM$ has a $\wnabla$-filtration provided that
$\Ext^1_\wA(\wDelta(\mu),\wM)=0$, for all $\mu\in\Lambda$. Or, observe that the argument above
works for right as well as left $\wA$-modules, and then take a linear dual.\end{rem}

\section{Appendix II} 
\setcounter{equation}{0}

We assume the notation and setup of the beginning of \S4.1, and let $\Gamma$
be a finite poset of regular weights. Assume that the ``LCF holds for $\Gamma$" in the sense that $\Delta^p(\gamma) =\rDelta(\gamma)$, for all $\gamma\in\Gamma$. We will prove a substitute for Lemmas \ref{preliminaryLCF} and \ref{numlemma} under
this assumption, without requiring that the LCF holds for all dominant weights in the Jantzen region. 
More precisely, we claim
that $(\widetilde\soc\, Q^\sharp(\gamma))^\Gamma\cong\rnabla(\gamma)$, for all $\gamma\in\Gamma$.

The condition $\Delta^p(\gamma)=\rDelta(\gamma)$ is equivalent to the statement $\Delta^p(\gamma_0)
=\rDelta(\gamma_0)$, if $\gamma=\gamma_0+p\gamma_1$, with $\gamma_0\in X_1(T)$, $\gamma_1\in
X(T)_+$. Of course, $\Delta^p(\gamma_0)=L(\gamma_0)$.  We also note the obvious fact that $\Delta^p(\gamma)|_\fa$ is a direct sum of copies of the irreducible $\fa$-module $L(\gamma_0)|_\fa$. 

The module $ Q^\sharp(\gamma)$ is dual to $P^\sharp(\gamma)/(\wrad\fa)P^\sharp(\gamma)$, which is the reduction mod $\pi$ of $\wP^\sharp(\gamma)/\wrad\wP^\sharp(\gamma)$, as noted in the proof of
Lemma \ref{numlemma}. The latter
module may be interpreted as $(\gr\wP^\sharp(\gamma))_0$, and $\gr\wP^\sharp(\gamma)$ has
a graded $\gr\wDelta$-filtration. Thus, $(\gr\wP^\sharp(\gamma))_0$ has a $\wrDelta$-filtration.  All $\wrDelta(\nu)$-sections, for $\nu\in\Gamma$, may be assumed to occur together at the top. Tracing
back to $\widetilde\soc Q^\sharp(\gamma)$ gives that the desired equality
 $(\widetilde Q^\sharp(\gamma))^\Gamma=\rnabla(\gamma)$ holds if and only if the only section $\wrDelta(\nu)$, $\nu\in\Gamma$,
 occurring in the $\wrDelta$-filtration of $\wP^\sharp(\gamma)/\wrad\wP^\sharp(\gamma)$ is the evident top
 term $\wrDelta(\gamma)$. Equivalently, the head of $\wP^\sharp(\gamma)_{K,\Gamma}$ is $L_K(\gamma)$. We now prove this assertion. 
 
 Write $\wP^\sharp(\gamma)_K$ as a direct sum of indecomposable $\wU_{\zeta,K}$-modules. Each is,
 of course, a projective indecomposable module, and so has the form $P_K(\omega)$ for some $\omega\in \Xreg $. Since
 there is a nonzero homomorphism $\wP^\sharp(\gamma)\subseteq\wP^\sharp(\gamma)_K\to
 P_K(\omega)/\rad P_K(\omega)\cong L_K(\omega)$, $L(\gamma)$ must appear in the reduction mod
 $\pi$ of any full lattice in $L_K(\omega)$. Hence, $\gamma\leq \omega$. 
 
 Suppose that $\omega\in\Gamma$. Then $\rDelta(\omega_0)=L(\omega_0)$ is the irreducible $G$-module of highest weight $\omega_0$, hence is an irreducible $G_1$-module, and the same is true for $L(\gamma_0)$. Thus, $L(\gamma_0)\cong L(\omega_0)$ as $G_1$-modules, and $\gamma_0=\omega_0$. Moreover, the highest weight $2(p-1)\rho+w_0\omega_0$ of $Q(\gamma_0)$ is a weight of $Q_K(\gamma_0)=P_K(\gamma_0)$, so $2(p-1)\rho+w_0\omega_0+p\omega_1\leq 2(p-1)\rho+w_0\gamma_0+p\gamma_1$, the highest weight of $Q^\sharp(\gamma)$. Therefore $\omega_1\leq\gamma_1$, since $\omega_0=\gamma_0$.
 Also, $\omega_1\geq \gamma_1$, since $\omega\geq \gamma$, as already shown. Thus, $\omega=\gamma$.
 
 Consequently, all irreducible modules in the head $\wP^\sharp(\gamma)_{K,\Gamma}$ are isomorphic
 to $L_K(\gamma)$. However, $L_K(\gamma)$ can appear only once in the head, since $\wP^\sharp(\gamma)$ is a homomorphic image of the projective $\wA_\Lambda$-cover $\wP(\gamma)$ of $\wDelta(\gamma)$, and its own $\wDelta$-filtration is part of a $\wDelta$-filtration of $\wP(\gamma)$. Hence, $L_K(\gamma)$ occurs only once in the head of $\wP^\sharp(\gamma)$, with all other composition factors of the head indexed by
 higher weights. This completes the proof of the claimed equality.
 
\medskip The claim can now serve as a substitute for   Lemma \ref{numlemma}. This allows both
 Theorems \ref{thm4.4} and \ref{firstmainnum}  to go through, provided the highest weights of composition factors of
 the module $M$ belong to a poset $\Gamma$ for which the LCF holds. Therefore, the following
 analog of Theorem \ref{MainTheorem} holds.
 
 \begin{thm}\label{thm7.1}Assume that $p\geq 2h-2$ is odd. Suppose that $\lambda\in \Xreg $ and that
 the LCF holds for the poset $\Gamma:=\{\gamma\in \Xreg \,|\,\gamma<\lambda\}$. Then
 each section $\Delta_{\widetilde s}(\lambda)$, $s\geq 0$, viewed as a rational $G$-module, has a
 $\rDelta$-filtration. Each standard module $\Delta(\gamma)$, 
with $\gamma\in X(T)_+$ satisfying  $\gamma \leq \lambda$, also has a $\rDelta$-filtration.
 \end{thm}

We remind the reader that $\rDelta(\lambda')= \Delta^p(\lambda')$ for
$\lambda' \in \Gamma$. A similar statement holds for $\gamma' \in X(T)_+$ and
and $\gamma'\leq \lambda'$, by a translation argument, which we leave
to the reader.

\end{document}